\documentclass[11pt]{article}
\usepackage{amsmath,amsthm,amsfonts,amssymb,amscd}
\usepackage{epsfig,epstopdf}
\usepackage{graphicx}
\usepackage[svgnames]{xcolor}
\usepackage{tabularx,caption}
\usepackage{array}
\usepackage{booktabs}
\usepackage{xr}
\usepackage{caption}
\usepackage[normalem]{ulem}

\theoremstyle{plain}
\newtheorem{theorem}{Theorem}[section]
\newtheorem*{theorem*}{Theorem}
\newtheorem{lemma}{Lemma}[section]

\newtheorem{proposition}{Proposition}[section]

\theoremstyle{definition}
\newtheorem{definition}{Definition}[section]
\newtheorem*{definition*}{Definition}
\newtheorem{example}{Example}[section]

\newtheorem{non-example}{Non-example}[section]

\theoremstyle{remark}
\newtheorem{remark}{Remark}[section]
\newtheorem*{remark*}{Remark}

\newcommand{\MSD}{\textnormal{TAMSD}}
\newcommand{\bbZ}{{\Bbb Z}}
\newcommand{\bbR}{{\Bbb R}}
\newcommand{\bbN}{{\Bbb N}}
\newcommand{\bbC}{{\Bbb C}}

\newcommand{\bbP}{{\Bbb P}}

\newcommand{\bbE}{{\Bbb E}}

\newcommand{\cov}{{\mathrm{Cov}}}
\newcommand{\cor}{{\mathrm{Corr}}}
\newcommand{\Var}{{\mathrm{Var}\hspace{0.5mm}}}

\newcommand{\abs}[1]{\left|#1\right|}

\newcommand{\imag}{\textbf{i}}

\begin{document}

\title{{Asymptotic theory for the detection of mixing in anomalous diffusion}
\author{Kui Zhang and Gustavo Didier \\ Mathematics Department\\ Tulane University}
}

\maketitle

\begin{abstract}
In this paper, we develop asymptotic theory for the mixing detection methodology proposed by M.\ Magdziarz and A.\ Weron [\textit{Physical Review E}, 84:051138 (2011)]. The assumptions cover a broad family of Gaussian stochastic processes including fractional Gaussian noise and the fractional Ornstein-Uhlenbeck process. We show that the asymptotic distribution and convergence rates of the detection statistic may be, respectively, Gaussian or non-Gaussian and standard or nonstandard depending on the diffusion exponent. The results pave the way for mixing detection based on a single observed sample path and by means of robust hypothesis testing.
\end{abstract}


\section{Introduction}


Let $Y = \{Y_n\}_{n \in \bbZ}$ be a $\bbR$-valued, strictly stationary stochastic process, and let ${\mathcal F}_Y$ be the $\sigma$-algebra generated by $Y$. Also let $\{T^{n}\}_{n \in \bbZ}$ be the group of shift transformations induced by $Y$. We say $Y$ is \textit{mixing} if
$$
\lim_{n \rightarrow \infty} \bbP(A \cap T^n(B)) = \bbP(A)\bbP(B), \quad A,B \in {\mathcal F}_Y.
$$
In other words, the past and the future of $Y$ are asymptotically independent \cite{peligrad:1986,samorodnitsky:2016}. For $n \in \bbN$, let $E(n) = \bbE e^{\imag (Y_n - Y_0)} - |\bbE e^{\imag Y_0}|^2 \in \bbC$. If the strictly stationary process $Y$ is infinitely divisible and its L\'{e}vy measure has no atoms in $2\pi\bbZ$, then it is mixing if and only if
\begin{equation}\label{e:lim_E(n)=0}
\lim_{n \rightarrow \infty} E(n) = 0
\end{equation}
\cite{maruyama:1970,gross:1994,cambanis:podgorski:weron:1995,rosinski:zak:1996}. If, in addition, $Y$ is Gaussian with autocovariance $\gamma_Y$, condition \eqref{e:lim_E(n)=0} is equivalent to
\begin{equation}\label{e:lim_E(n)=0_Gaussian}
\lim_{n \rightarrow \infty} \gamma_Y(n) = 0.
\end{equation}


Recall that a stochastic process ${\mathcal Y}$ is said to be \textit{anomalously diffusive }when, for some $\alpha \neq 1$, its ensemble mean squared displacement (MSD) satisfies the relation
\begin{equation}\label{e:EX2(t)=Dt(alpha)}
\bbE \hspace{0.5mm}{\mathcal Y}^2(t) \propto \sigma^2 t^\alpha, \quad \sigma^2 > 0, \quad t \geq 0.
\end{equation}
Anomalously diffusive behavior has been consistently observed in several physical contexts, such as in turbulence \cite{kolmogorov:1941,friedlander:topper:1961,shiryaev:1999} and single particle tracking \cite{mason:weitz:1995,suh:dawson:hanes:2005,matsui:wagner:hill:etal:2006,lai:wang:cone:wirtz:hanes:2009,barkai:garini:metzler:2012,burnecki:muszkieta:sikora:weron:2012,metzler:jeon:cherstvy:2016}. Suppose $N \in \bbN$ observations of the (strictly stationary) increment process $Y$ of ${\mathcal Y}$ are available. For $n \in \bbN$, consider the sample analogue of $E(n)$, namely,
\begin{equation}\label{e:E(n)}
\widehat{E}_N(n) = \frac{1}{N-n+1} \sum_{k=0}^{N-n} \exp\{ \imag [Y(n+k) - Y(k)] \} - \abs{\frac{1}{N+1} \sum_{k=0}^{N} \exp\{ \imag Y(k) \}}^2, \quad n < N.
\end{equation}
If $Y$ is mixing, for large $n$ condition \eqref{e:lim_E(n)=0} suggests that
\begin{equation}\label{e:E(n)rightarrow0}
\widehat{E}_N(n) \approx 0.
\end{equation}
In light of \eqref{e:E(n)rightarrow0}, the statistic $\widehat{E}_N(n)$ is proposed in \cite{magdziarz:weron:2011} as the basis for a protocol for the detection of mixing in anomalous diffusion (for discussions, applications and extensions, see \cite{sokolov:2012,lanoiselee:grebenkov:2016,weron:burnecki:akin:sole:balcerek:tamkun:krapf:2017,slkezak:metzler:magdziarz:2019}). Nevertheless, the fundamental convergence property \eqref{e:E(n)rightarrow0} does not per se provide a means of \textit{quantifying the uncertainty} involved in the use of the statistic $\widehat{E}_N(n)$. Namely, when can we say $\widehat{E}_N(n)$ is close enough to zero -- i.e., how strong is the evidence of mixing -- starting from a single sample path?

We say that a cumulative distribution function (c.d.f.) $F$ gives the \textit{asymptotic distribution }of a sequence of random variables $\{W_N\}_{N \in \bbN}$ if the c.d.f.\ $F_N(x)$ of $W_N$ converges to $F(x)$ at every $x \in \bbR$ where $F$ is continuous. Results on convergence in distribution such as the classical central limit theorem (CLT) have a number of interesting consequences. Typically, statements are mathematically \textit{robust}, i.e., they hold for a multitude of models. Moreover, they naturally \textit{quantify} the accuracy of data analysis protocols by providing confidence intervals and hypothesis tests with explicit \textit{error margins}. In addition, such error margins depend on the sample size $N$ in a way that characterizes the convergence rate of the test statistics (e.g., $\sqrt{N}$ for the sample mean under the CLT). In particular, this allows experimentalists to compute the sample sizes associated with the desired uncertainty levels before carrying out their experiments.

This paper is a contribution to the quantification of the uncertainty involved in the use of data analysis protocols based on $\widehat{E}_N(n)$. The asymptotic distribution of $\widehat{E}_N(n)$ under various families of strictly stationary stochastic processes $Y= \{Y_n\}_{n \in \bbZ}$ is a broad mathematical problem that has remained open in the anomalous diffusion literature. The large sample behavior of $\widehat{E}_N(n)$ is expected to depend not only on the correlation structure of $Y$, but also on its marginal distributions. In this paper, we establish the asymptotic distribution of $\widehat{E}_N(n)$ for fixed $n$ as $N \rightarrow \infty$, starting from a large class of stationary-increment Gaussian anomalous diffusion processes (see Theorem \ref{t:E(n)dist}, and also Propositions \ref{prop:E_1(n)} and \ref{prop:E_2(n)}). We assume the correlation structure of the increments displays fractional memory, i.e., its first and second order increments have hyperbolic decay. Such assumptions are mild and cover cornerstone models such as fractional Gaussian noise and the fractional Ornstein-Uhlenbeck process (see Examples \ref{ex:fGn} and \ref{ex:fOU} and Lemma \ref{lem:gammaZ}). In particular, condition \eqref{e:lim_E(n)=0_Gaussian} is satisfied.

Table \ref{table:asympt_dist} schematically portrays the main result of the paper. Developing the asymptotic distribution of $\widehat{E}_N(n)$ involves dealing with the nature of dependence in anomalous diffusion, as well as with the functional form of the statistic $\widehat{E}_N(n)$. The presence of fractional memory may even give rise to unconventional convergence rates and non-Gaussian distributions. Such phenomenon is grounded in the so-named \textit{central }and \textit{non-central limit theorems} for sequences of random variables exhibiting dependence (e.g., \cite{rosenblatt:1961,taqqu:1975,taqqu:1979,dobrushin:major:1979,major:1981,giraitis:surgailis:1985,guyon:leon:1989,giraitis:surgailis:1990}). In the context of anomalous diffusion, the effect of the presence of fractional memory also appears conspicuously in the asymptotic distribution of the time-averaged MSD ($\MSD$; \cite{didier:zhang:2017,zhang:crizer:schoenfisch:hill:didier:2018}).

In the subdiffusive ($0 < \alpha < 1$), diffusive ($\alpha = 1$) and weakly superdiffusive regimes ($1 < \alpha < 3/2$), knowledge of the asymptotic distribution of $\widehat{E}_N(n)$ paves the way for the construction of \textit{asymptotically valid} (and \textit{Gaussian}) mixing detection procedures based on quantifiable error margins and assuming a \textit{single} sample path is available. In particular, it naturally leads to \textit{robust hypothesis testing} (see Section \ref{s:discussion} for a discussion; see also Section IV in \cite{magdziarz:weron:2011}). The results also shed new light on the use of $\widehat{E}_N(n)$  in the strongly superdiffusive regime $3/2 < \alpha < 2$, due to the fact that non-Gaussian limits are involved. This brings the literature on anomalous diffusion close to recent important efforts in the Probability and Statistics literature towards understanding and modeling fractional non-Gaussian phenomena (e.g., \cite{nourdin:nualart:tudor:2010,chronopoulou:viens:tudor:2011,veillette:taqqu:2013,pipiras:taqqu:2017}).

\vspace{3mm}
\setlength{\heavyrulewidth}{1.5pt}
\setlength{\abovetopsep}{4pt}
\begin{table}[!htbp]
\centering
\begin{tabular}{*6c}
\toprule
parameter range &  \multicolumn{2}{c}{rate of convergence} & \multicolumn{3}{c}{asymptotic distribution} \\
{}  & $ {\mathfrak R}\widehat{E}_N(n) $   & $ {\mathfrak I}\widehat{E}_N(n) $   &   $ {\mathfrak R}\widehat{E}_N(n) $   & $ {\mathfrak I}\widehat{E}_N(n) $ &  joint \\\hline
$0 < \alpha < 3/2$    &  standard        &   standard                         & Gaussian  & Gaussian & Gaussian\\
 $3/2 < \alpha < 2$   &  nonstandard     &  standard                          & non-Gaussian  & Gaussian & non-Gaussian\\
\bottomrule
\end{tabular}
\caption{Asymptotic behavior of $\widehat{E}_N(n)$ (see Theorem \ref{t:E(n)dist}).} \label{table:asympt_dist}
\end{table}
\vspace{3mm}

The paper is organized as follows. In Section \ref{s:background}, we recap some fundamental concepts such as fractional memory, Hermite polynomials and Hermite processes.  In Section \ref{s:main_results}, we establish the main result, namely, the study of the asymptotic distribution of $\widehat{E}_N(n)$. In Section \ref{s:conclusion_and_open_problems}, we provide conclusions and a list of open problems, including the asymptotic distribution of $\widehat{E}_N(n)$ starting from non-Gaussian measurements. The appendix contains all proofs (Section \ref{app:mixingproofs}) as well as a review of the basal framework of central and non-central limit theorems (Section \ref{s:central_non-central}).

\section{Background}\label{s:background}

Developing the asymptotic distribution of the statistic $\widehat{E}_N(n)$ involves three main mathematical elements.
\begin{itemize}
\item [$(i)$] a Gaussian stationary sequence $\{X_n\}_{n \in \bbN}$ with \textit{fractional memory}.
\item [$(ii)$] \textit{Hermite polynomials} $H_n(x)$, based on which we can decompose deterministic functions $G_r$ such that $\bbE G^2_r(X_n) < \infty$, $r = 1,\hdots,R$.
\item [$(iii)$] \textit{Hermite stochastic processes}, which appear in the asymptotic distribution of the standardized random vector processes
\begin{equation}\label{e:V_N(t)}
  {\mathbf V}_N(t) = \bigg\{\frac{1}{A_r(N)} \sum_{n=1}^{[Nt]}\big(G_r(X_n) - \bbE G_r(X_n)\big)\bigg\}_{r= 1,\hdots,R},\quad t\geq 0,
\end{equation}
for appropriate sequences $\{A_r(N)\}_{N \in \bbN}$, $r = 1,\hdots,R$, $R \in \bbN$.
\end{itemize}

For the reader's convenience, we now briefly describe each of these elements.

\subsection{Gaussian stationary sequences with fractional memory}

In this paper, we consider anomalous diffusion processes whose increments are Gaussian stationary sequences $\{X_n\}_{n\in \bbZ} $ with covariance function
\begin{equation}\label{e:gamma_X}
  \gamma_X(k) = L(k) k^{\alpha-2} \equiv L(k) k^{2d-1}, \quad \alpha \in (0,2), \quad k\in \bbN \cup \{0\},
\end{equation}
and satisfying $\bbE X_n = 0$, $\Var X_n = 1$, $n\in\bbN$. In \eqref{e:gamma_X}, $L$ is a \emph{slowly varying function} at infinity, namely, it is positive on $[c,\infty)$ with $c \geq 0$ and, for any $a > 0$,
\begin{equation}\label{e:L_2def}
  \lim_{u\rightarrow \infty} \frac{L(au)}{L(u)} = 1.
\end{equation}

Two canonical examples of such sequences are provided next.
\begin{example}\label{ex:fGn}
Let $\{B_H(t)\}_{t \in \bbR}$ be a standard fBm. A (standard) \textit{fractional Gaussian noise} (fGn) $X$ is defined as the increment process
$$
X(n) = B_H(n+1) - B_H(n), \quad n \in \bbN \cup \{0\},
$$
where $\bbE X^2(n) = 1$. The autocovariance function of $X$ is given by
$$
\gamma_X(k) = \frac{1}{2} \{ \abs{k+1}^{\alpha} - 2 \abs{k}^{\alpha} + \abs{k-1}^{\alpha} \}, \quad k\in\bbZ,
$$
where the diffusion exponent satisfies
\begin{equation}\label{e:alpha=2H}
\alpha = 2H, \quad 0<\alpha<2.
\end{equation}
Furthermore, we can rewrite the autocovariance function $\gamma_Y(k)$ in the form \eqref{e:gamma_X}, where
\begin{equation}\label{e:L_2(k)fGn}
  L(k) = \frac{k^2}{2} \bigg\{\abs{1 + \frac{1}{k}}^{\alpha} - 2 +  \abs{1 - \frac{1}{k}}^{\alpha}\bigg\}.
\end{equation}
In fact, by a Taylor expansion of order 2 for $\abs{1 \pm x}^{\alpha -2}$, it can be shown that $L(k)$ is a slowly varying function (see \eqref{e:L_2def}). Therefore, $X$ satisfies condition \eqref{e:gamma_X}.
\end{example}

\begin{example}\label{ex:fOU}
The \textit{fractional Ornstein-Uhlenbeck process} (fOU) $\{V(t)\}_{t \geq 0}$ is the almost surely continuous solution to the fBm-driven Langevin equation
\begin{equation} \label{e:fOU_SDE}
dV(t) = - \lambda V(t) dt + \sigma \hspace{0.5mm}dB_{H}(t), \quad t \geq 0, \quad \lambda,\sigma >0, \quad H \in (0,1) \backslash \Big\{\frac{1}{2}\Big\}
\end{equation}
(see \cite{prakasarao:2010,magdziarz:weron:2011b}). It can be shown that, for $t \geq 0$ and large $s$,
$$
\gamma_{V}(s) = \textnormal{Cov}(V(t),V(t+s) ) = \frac{\sigma^2}{2\lambda^2}  \alpha(\alpha-1) s^{\alpha-2} + O(s^{\alpha-4}),
$$
where $\alpha$ is as in \eqref{e:alpha=2H} (see expression \eqref{e:gamma_y(k)_fOU}). Therefore, the dependent sequence $\{V(n)\}_{n \in \bbN \cup \{0\}}$ satisfies condition \eqref{e:gamma_X}.
\end{example}

\subsection{Hermite polynomials}

In the following two definitions, we describe the notions of Hermite polynomial and Hermite rank.

\begin{definition}
The \emph{Hermite polynomial} of order 0 is $H_0(x) \equiv 1$, and the \emph{Hermite polynomial of order $n \in\bbN$} is defined by
$$
H_n(x) = (-1)^n e^{x^2/2} \frac{d^n}{dx^n} e^{-x^2/2}, \quad x\in\bbR.
$$
\end{definition}
In particular,
$$
H_1(x) = x,\quad H_2(x) = x^2 - 1.
$$
Let $\phi(dx) = (1/\sqrt{2\pi}) e^{-x^2/2} \hspace{0.5mm}dx$ be the probability measure on $\bbR$ associated with a standard normal random variable $Z$ and let $L^2(\phi)$ be the space of measurable, square-integrable functions with respect to $\phi(dx)$. Then, $G\in L^2(\bbR,\phi)$ if and only if $\bbE G^2(Z) < \infty$. Any $G\in L^2(\bbR,\phi)$ has a series expansion in Hermite polynomials
\begin{equation}\label{e:g_k}
G(x) = \sum_{m=0}^{\infty} g_m H_m(x), \quad   g_m = \frac{1}{m!} (G, H_m)_{L^2(\bbR, \phi)}, \quad m \in \bbN \cup \{0\}
\end{equation}
(see \cite{pipiras:taqqu:2017}, chapter 5).
\begin{definition}\label{def:hermiterank}
Let $G\in L^2(\bbR,\phi)$ and let $g_k$, $k\geq 0$, be the coefficients in its Hermite expansion. The \emph{Hermite rank} $r$ of $G$ is defined as the smallest index $k\geq 1$ for which $g_k \neq 0$, that is,
$$
r = \min\{ k\geq 1 : g_k \neq 0\}.
$$
\end{definition}

In view of the complex exponential appearing in \eqref{e:E(n)}, we are particularly interested in trigonometric functions $G_r$. For this reason, the following example will be useful throughout the paper.
\begin{example}
The Hermite ranks of the trigonometric functions $\sin(x)$ and $\cos(x)$ are 1 and 2, respectively. In fact, note that
$$
\int_{\bbR}^{} \sin(x) H_1(x) \frac{e^{-x^2/2}}{\sqrt{2\pi}} dx = \frac{1}{\sqrt{e}}
$$
and
\begin{equation}\label{e:mixingcosh1h2}
  \int_{\bbR}^{} \cos(x) H_1(x) \frac{e^{-x^2/2}}{\sqrt{2\pi}} dx = 0, \quad \frac{1}{2!}\int_{\bbR}^{} \cos(x) H_2(x) \frac{e^{-x^2/2}}{\sqrt{2\pi}} dx = - \frac{1}{2\sqrt{e}}.
  \end{equation}
So, for $m \in \bbN \cup \{0\}$, hereinafter the expressions
\begin{equation}\label{e:g_cossin}
  g_{\cos,m} = \frac{1}{m!} \int_{\bbR} \cos(x) H_m(x) \frac{e^{-x^2/2}}{\sqrt{2\pi}} dx, \quad
g_{\sin,m} = \frac{1}{m!} \int_{\bbR} \sin(x) H_m(x) \frac{e^{-x^2/2}}{\sqrt{2\pi}} dx,
\end{equation}
denote the $m$-th coefficients of the series expansions of $\cos(x)$ and $\sin(x)$, respectively, in Hermite polynomials.
\end{example}

\subsection{Hermite stochastic processes}

It turns out that, in general, the limit in distribution of \eqref{e:V_N(t)} is given by a so-named \textit{Hermite stochastic process} (\cite{pipiras:taqqu:2017}, chapter 4). This class of processes is only Gaussian under particular conditions on $G_r$ and on the covariance structure of $Y$. The best-known Hermite process is fractional Brownian motion (fBm), i.e., the only Gaussian, self-similar, stationary-increment process \cite{embrechts:maejima:2002}. The general definition of Hermite process is given next.
\begin{definition}\label{def:Z_H(t)}
Let $k\geq 1$ be an integer and suppose $H\in(1/2,\hspace{0.5mm}1)$. 
The \emph{Hermite process} $Z_H^{(k)} = \{Z_H^{(k)}(t)\}_{t\in \bbR}$ of order $k$ can be defined by means of the harmonizable representation
\begin{equation}\label{e:hermiteprocess}
  \{ Z_H^{(k)}(t) \}_{t\in\bbR} \overset{{\mathcal L}}{=} \bigg\{ b_{k,H_0} \int_{\bbR^k}^{''} \frac{e^{\imag t(x_1 + \hdots + x_k)} - 1}{\imag (x_1 + \hdots + x_k)} \prod_{j=1}^{k} \abs{x_j}^{\frac{1-H}{k}-\frac{1}{2}}
\widetilde{B}(dx_1) \cdots \widetilde{B}(dx_k) \bigg\}_{ t\in \bbR}.
\end{equation}
In \eqref{e:hermiteprocess},
$$
b_{k,H} = \bigg( \frac{H(2H-1)}{k! \big[2\Gamma\big(\frac{2(1-H)}{k}\big) \sin((\frac{1}{2} - \frac{1-H}{k})\pi)\big]^k} \bigg)^{1/2},
$$
and $\widetilde{B}(dx)$ is a $\bbC$-valued Gaussian random measure on $\bbR$ such that $\widetilde{B}(-dx) = \overline{\widetilde{B}(dx)}$ and $\bbE |\widetilde{B}(dx)|^2 =  dx$ (see, for instance, Section 6 in \cite{taqqu:1979}). The double prime on the integral \eqref{e:hermiteprocess} indicates that the integration domain excludes the diagonals where $x_i = \pm x_j$, $i \neq j$. The Hermite process $Z_H^{(k)}$ is called \emph{standard} if $\bbE (Z^{(k)}_H(1))^2 = 1$.
\end{definition}

The following example sheds more light on the notion of Hermite stochastic process.
\begin{example}
The Hermite process of order $k=1$ is fBm. The Hermite process of order $k=2$ is called the \textit{Rosenblatt process} \cite{taqqu:1975}. All standard Hermite processes have covariance function
\begin{equation}\label{e:mixing_fbmcov}
  \bbE Z_H^{(k)}(s)Z_H^{(k)}(t) = \frac{1}{2} \{\abs{s}^{2H} + \abs{t}^{2H} - \abs{s-t}^{2H}\}, \quad s,t \in \bbR,
\end{equation}
namely, the classical covariance function of fBm.

Furthermore, it is convenient to reexpress the Hermite process \eqref{e:hermiteprocess} as follows. For any appropriate integrand $f$, let
\begin{equation}\label{e:widehatI_k(f)}
 \widehat{I}_{k}(f) = \int_{\bbR^k}^{''} f(x_1,\hdots,x_k) \widetilde{B}(dx_1) \cdots \widetilde{B}(dx_k).
\end{equation}
Then,
$$
\{ Z_H^{(k)}(t) \}_{t\in\bbR} \overset{\textnormal{f.d.d.}}{=}  \{ \widehat{I}_{k}(f_{H,k,t})  \}_{ t\in \bbR},
$$
where $\overset{\textnormal{f.d.d.}}{=}$ denotes the equality of finite dimensional distributions and
\begin{equation}\label{e:f_Hkt}
  f_{H,k,t} (x_1,\hdots,x_k) := b_{k,H} \hspace{0.5mm} \frac{e^{\imag t(x_1+\hdots+x_k)} - 1}{\imag (x_1 + \hdots + x_k)}\prod_{j=1}^{k} \abs{x_j}^{\frac{1-H}{k} - \frac{1}{2}}.
\end{equation}
\end{example}

\section{Main results}\label{s:main_results}

\subsection{The asymptotic distribution of $\widehat{E}_N(n)$}

Let $Y$ be a centered Gaussian stationary process with autocovariance function $\gamma_Y$ and satisfying $\bbE Y^2(0) = 1$. Suppose the observations $\{Y(k)\}_{k=0,\hdots,N}$ are available, and consider the statistic $\widehat{E}_N(n)$ defined by \eqref{e:E(n)}. For $z \in \bbC$, let $\mathfrak{R}[z]$ and $\mathfrak{I}[z]$ be its real and imaginary parts, respectively. Note that $\widehat{E}_N(n)$ can be decomposed as $\widehat{E}_{N,1}(n) - \widehat{E}_{N,2}$, where
$$
\bbC \ni  \widehat{E}_{N,1}(n) = \frac{1}{N-n+1} \sum_{k=0}^{N-n} \cos\big(Y(n+k) - Y(k)\big)
$$
\begin{equation}\label{e:E1}
+ \hspace{0.5mm}\imag \hspace{0.5mm}\frac{1}{N-n+1} \sum_{k=0}^{N-n} \sin\big(Y(n+k) - Y(k)\big),
\end{equation}
\begin{equation}\label{e:E2}
 \bbR \ni \widehat{E}_{N,2} = \abs{\frac{1}{N+1} \sum_{k=0}^{N} \cos Y(k)}^2 + \abs{\frac{1}{N+1} \sum_{k=0}^{N} \sin Y(k)}^2.
\end{equation}
In particular,
$$
\mathfrak{R} \widehat{E}_N(n) = \mathfrak{R} \widehat{E}_{N,1}(n) - \widehat{E}_{N,2},\quad
\mathfrak{I} \widehat{E}_N(n) = \mathfrak{I} \widehat{E}_{N,1}(n).
$$
Theorem \ref{t:E(n)dist}, stated next, is the main result of this paper. There, we establish the asymptotic distribution of $\widehat{E}_N(n)$, for $n\in\bbN$ as $N \rightarrow \infty$, starting from stationary-increment Gaussian anomalous diffusion processes. In particular, the autocovariance of the increment process $Y$ satisfies condition \eqref{e:lim_E(n)=0_Gaussian}, i.e., $Y$ is mixing. The theorem contains two main statements, corresponding to the asymptotic distributions of $\mathfrak{R}\widehat{E}_N(n)$ and $\mathfrak{I} \widehat{E}_N(n)$. In the first one, irrespective of the value of $\alpha$, the asymptotic distribution of $\mathfrak{I} \widehat{E}_N(n)$ is claimed to be Gaussian and the convergence rate is standard ($\sqrt{N+1}$). In the second one, the asymptotic distribution of $\mathfrak{R}\widehat{E}_N(n)$ is claimed to depend on the value of $\alpha$. When $0<\alpha <3/2$, the convergence rate is standard ($\sqrt{N+1}$) and the limiting distribution is Gaussian. By contrast, when $3/2<\alpha <2$, the convergence rate is nonstandard (nearly $(N+1)^{2-\alpha}$) and the limiting distribution is non-Gaussian.
\begin{theorem}\label{t:E(n)dist}
Let $\{Y(n)\}_{n \in \bbZ}$ be a centered Gaussian stationary process with $\bbE Y^2(0)=1$. Let $\gamma_Y$ be its autocovariance function, which is assumed to satisfy the decay condition
\begin{equation}\label{e:gammaY(k)=O(k^(alpha - 2))}
\gamma_Y(k) = O(k^{\alpha - 2}), \quad k \rightarrow \infty, \quad  0<\alpha<2.
\end{equation}
Let $\{\widetilde{Z}_n(k)\}_{k\in \bbZ}$ be an associated normalized increment process given by
\begin{equation}\label{e:defZn}
  \widetilde{Z}_n(k) = \sigma(n)[Y(n+k) - Y(k)], \quad k \in \bbN,
\end{equation}
where $\sigma(n)$ is defined so that $\Var \widetilde{Z}_n(k) = 1$. Let $\gamma_{\widetilde{Z}_n}$ be its autocovariance function, which is assumed to satisfy the decay condition
\begin{equation}\label{e:gammaZtilde(k)=O(k^(alpha-3))}
\gamma_{\widetilde{Z}_n}(k) = O(k^{\alpha -3}), \quad k \rightarrow \infty, \quad  0<\alpha<2.
\end{equation}
\begin{itemize}
  \item [($i.1$)] When $0<\alpha<3/2$, if
  \begin{equation}\label{e:cross-cov_summability}
  \sum^{\infty}_{\ell=-\infty}\Big|\sum_{m = 2}^{\infty} g_{1,n,m} g_{\cos,m} m!  (\gamma_{\widetilde{Z}_n,Y}(\ell))^m \Big| < \infty,
  \end{equation}
  then
  \begin{equation}\label{e:weak_limit_B1(1)_B3(1)}
  \sqrt{N+1} \bigg( \mathfrak{R} \widehat{E}_N(n) - \bbE \cos\bigg(\frac{\widetilde{Z}_n(0)}{\sigma(n)}\bigg) + \abs{\bbE \cos Y(0)}^2 \bigg)
\overset{d}{\rightarrow} \sigma_{1,n} B_1(1) + 2 \hspace{0.5mm}\big(\bbE \cos Y(0)\big) \hspace{1mm}\sigma_3 B_3(1),
\end{equation}
as $N\rightarrow \infty$, where $\sigma_{1,n}$ is defined in \eqref{e:sigma_12}, $\sigma_3$ is defined in \eqref{e:sigma3}, $B_1(1)$ and $B_3(1)$ are two standard normal distributions with correlation
\begin{equation}\label{e:corrB1B3}
  \cor (B_1(1), B_3(1)) = - \frac{\sum^{\infty}_{\ell=-\infty}\sum_{m = 2}^{\infty} g_{1,n,m} g_{\cos,m} m!  (\gamma_{\widetilde{Z}_n,Y}(\ell))^m }{\sigma_{1,n} \sigma_3}.
\end{equation}
In \eqref{e:corrB1B3}, $g_{1,n,m}$ and $g_{\cos,m}$ are the coefficients in the Hermite expansion \eqref{e:g_k} of $\cos(x/\sigma(n))$ and $\cos(x)$, respectively, and $\gamma_{\widetilde{Z}_n,Y}$ is the cross-covariance function of $\widetilde{Z}_n(k+j)$ and $Y(j)$;
  \item [($i.2$)] when $3/2<\alpha<2$,
  $$
\frac{(N+1)^{2-\alpha}}{L(N+1)} \bigg( \mathfrak{R} \widehat{E}_N(n) - \bbE \cos\bigg(\frac{\widetilde{Z}_n(0)}{\sigma(n)}\bigg) + \abs{\bbE \cos(Y(0))}^2 \bigg)
$$
\begin{equation}\label{e:weak_limit_Re_E-hat_N(n)_3/2<alpha<2}
\overset{d}{\rightarrow} g_{\cos,2}\beta_{2,\alpha - 1} \widehat{I}_{2}(f_{\alpha - 1,2,1})
    + g_{\sin,1}^2 \beta_{1,\alpha/2}^2 \widehat{I}_{1}^2(f_{\alpha/2,1,1}),
\end{equation}
as $N\rightarrow \infty$, where $g_{\cos,2}$ and $g_{\sin,1}$ are given in \eqref{e:g_cossin}, $\beta_{2,\alpha - 1}$ and $\beta_{1,\alpha/2}$ are defined in \eqref{e:beta_kH}, $f_{H,k,t}$ is the kernel function defined in \eqref{e:f_Hkt}, $\widehat{I}_{k}(f)$ is defined in \eqref{e:widehatI_k(f)}, $L(N+1)$ is defined in \eqref{e:L_2(k)fGn};
  \item [($ii$)] for $0<\alpha<2$,
  \begin{equation}\label{e:weak_limit_B2(1)}
  \sqrt{N+1} \, \mathfrak{I} \widehat{E}_N(n) \overset{d}{\rightarrow} \sigma_{2,n} B_2(1),
  \end{equation}
as $N\rightarrow \infty$, where $\sigma_{2,n}$ is defined in \eqref{e:sigma_12}, and $B_2(1)$ is a standard normal distribution.
\end{itemize}
In the parametric range $0 < \alpha < 3/2$, the random variables $B_2(1)$ in \eqref{e:weak_limit_B2(1)} and $B_1(1)$ and $B_3(1)$ in \eqref{e:weak_limit_B1(1)_B3(1)} satisfy
\begin{equation}\label{e:Cov(B2,Bell(1))=0}
\cor \big(B_2(1),B_\ell(1)\big) = 0, \quad \ell = 1,3.
\end{equation}
\end{theorem}

\begin{example}
The assumptions of Theorem \ref{t:E(n)dist} are mild and cover broad classes of anomalous diffusion models. For example, expressions \eqref{e:sum_gamma_Z(n)}, \eqref{e:gamma_y(k)} and \eqref{e:gamma_y(k)_fOU} in Lemma \ref{lem:gammaZ} and expression \eqref{e:summability_cross-cov} in Lemma \ref{l:summability_cross-cov} show that Theorem \ref{t:E(n)dist} holds when $Y$ follows either a fGn or a fOU process.
\end{example}

\begin{example}
Suppose $\{Y(n)\}_{n \in \bbN \cup \{0\}}$ is a fGn with Hurst parameter $H \in (0,1)$. Then, the term $\sigma(n)$ as in expression \eqref{e:defZn} is given by
$$
\sigma(n) = \frac{1}{\sqrt{2-\abs{n-1}^{\alpha} + 2\abs{n}^{\alpha} - \abs{n+1}^{\alpha}}}, \quad n \in \bbN.
$$
\end{example}

The results encapsulated in Theorem \ref{t:E(n)dist} rely on first establishing two intermediary results, namely, Propositions \ref{prop:E_1(n)} and \ref{prop:E_2(n)}. In Proposition \ref{prop:E_1(n)}, we establish the limit law of the random vector $\big(\mathfrak{R}\widehat{E}_{N,1}(n), \mathfrak{I}\widehat{E}_{N,1}(n)\big)$, which turns out to be a bivariate Gaussian distribution. Since $\widehat{E}_{N,2}$ is real-valued, the asymptotic distribution of $\mathfrak{I}\widehat{E}_N(n)$ -- as described in Theorem \ref{t:E(n)dist} -- is the same as that of $\mathfrak{I}\widehat{E}_{N,1}(n)$. On the other hand (and keeping in mind that $\mathfrak{R} \widehat{E}_N(n) = \mathfrak{R} \widehat{E}_{N,1}(n) - \widehat{E}_{N,2}$), the stochastic processes involved in the expressions for $\mathfrak{R} \widehat{E}_{N,1}(n) $ and $\widehat{E}_{N,2}$ -- namely, $Y$ and $\widetilde{Z}_n$, respectively -- are distinct. Therefore, we can\textit{not} directly apply multivariate theorems (i.e., Theorems \ref{thm:central_multi} or \ref{thm:noncentral_multi}) to obtain the joint asymptotic distribution of $\mathfrak{R}\widehat{E}_{N,1}(n)$ and $\widehat{E}_{N,2}$. However, it is shown in Proposition \ref{prop:E_2(n)} that the asymptotic distribution of $\widehat{E}_{N,2}$ is Gaussian or non-Gaussian, depending on whether $0 < \alpha  < 3/2$ or $3/2 < \alpha < 2$, respectively. So, when $0 < \alpha <3/2$, $\mathfrak{R}\widehat{E}_{N,1}(n)$ and $\widehat{E}_{N,2}$ are both asymptotically normally distributed. In this case, such joint distribution is determined by the cross-covariance function of $\widetilde{Z}_n(k+j)$ and $Y(j)$ -- as claimed in Theorem \ref{t:E(n)dist}. On the other hand, when $3/2 < \alpha <2$, the convergence rate of $\widehat{E}_{N,2}$ is approximately $N^{2-\alpha}$, which is much slower than that of $\mathfrak{R}\widehat{E}_{N,1}(n)$. Thus, the limiting distribution of $\mathfrak{R}\widehat{E}_N(n)$ is still given by that of $\widehat{E}_{N,2}$ -- again as claimed in Theorem \ref{t:E(n)dist}. 


So, Proposition \ref{prop:E_1(n)} is stated next. In its proof, we make use of a general result on the asymptotic distribution of a sum of the form \eqref{e:V_N(t)} (for $R=1$) after standardization when the covariance function is $k$-th power absolutely summable (see Theorem \ref{thm:central}).
\begin{proposition}\label{prop:E_1(n)}
Let $Y$ be a centered Gaussian stationary process with $\bbE Y^2(0) = 1$. Let $\{\widetilde{Z}_n(k)\}_{k\in \bbZ}$ be the associated random sequence as defined by \eqref{e:defZn}, and suppose its autocovariance function satisfies \eqref{e:gammaZtilde(k)=O(k^(alpha-3))}.  Also, let $\widehat{E}_{N,1}(n)$ be its associated statistic \eqref{e:E1}. Then, as $N\rightarrow \infty$,
\begin{itemize}
  \item [($i$)] the real part of $\widehat{E}_{N,1}(n)$ satisfies
\begin{equation}\label{e:E_1(n)_real}
\sqrt{N-n+1} \bigg( \mathfrak{R}\widehat{E}_{N,1}(n) - \bbE \cos\bigg( \frac{\widetilde{Z}_n(0)}{\sigma(n)} \bigg) \bigg) \overset{d}{\rightarrow} \sigma_{1,n} B_1(1);
\end{equation}
  \item [($ii$)] and the imaginary part of $\widehat{E}_{N,1}(n)$ satisfies
\begin{equation}\label{e:E_1(n)_imag}
\sqrt{N-n+1} \, \mathfrak{I}\widehat{E}_{N,1}(n) \overset{d}{\rightarrow} \sigma_{2,n} B_2(1).
\end{equation}
\end{itemize}
In \eqref{e:E_1(n)_real} and \eqref{e:E_1(n)_imag},
\begin{equation}\label{e:sigma_12}
  \sigma_{r,n}^2 = \sum_{m=1}^{\infty} g_{r,n,m}^2 m! \sum_{\ell=-\infty}^{\infty} (\gamma_{\widetilde{Z}_n}(\ell))^m,\quad r=1,2,
\end{equation}
$g_{1,n,m}$ and $g_{2,n,m}$ are the coefficients in the Hermite expansion \eqref{e:g_k} of the functions
\begin{equation}\label{e:mixingG12n}
  G_{1,n}(x) = \cos(x/\sigma(n)),\quad G_{2,n}(x) = \sin(x/\sigma(n)),
\end{equation}
respectively, and $B_1(1)$ and $B_2(1)$ are standard normal random variables.
\end{proposition}

We now turn to Proposition \ref{prop:E_2(n)}, which provides the limit law of $\widehat{E}_{N,2}$ in \eqref{e:E2}. Recall that $\widehat{E}_{N,2}$ is the sum of the squares of the sample averages of $\cos Y(n)$ and $\sin Y(n)$. Therefore, in order to describe the asymptotic distribution of $\widehat{E}_{N,2}$, we need to develop the limiting joint distribution of $\frac{1}{N+1} \sum_{k=0}^{N} \cos Y(k)$ and $\frac{1}{N+1} \sum_{k=0}^{N} \sin Y(k)$. As discussed in Section \ref{s:background}, the convergence rate of \eqref{e:V_N(t)} is a function that depends on both the Hermite rank of $G$ and the parameter $d = \alpha/2 - 1/2$. Thus, the convergence rates of $\frac{1}{N+1} \sum_{k=0}^{N} \cos Y(k)$ and $\frac{1}{N+1} \sum_{k=0}^{N} \sin Y(k)$ may be distinct, in which case one dominates the other. Then, the claim of the proposition is a consequence of the asymptotic behavior of random vectors when their covariance functions are $k$-th power absolutely summable or not (see Theorems \ref{thm:central_multi} and \ref{thm:noncentral_multi}, respectively).
\begin{proposition}\label{prop:E_2(n)}
Let $Y$ be a centered Gaussian stationary process with $\bbE Y^2(0) = 1$, and suppose its autocovariance function $\gamma_Y$ satisfies the decay condition \eqref{e:gammaY(k)=O(k^(alpha - 2))}. Let $\widehat{E}_{N,2}$ be its associated statistic \eqref{e:E2}. Then,
\begin{itemize}
  \item [($i$)] when $0< \alpha < 3/2$,
  \begin{equation}\label{e:e_2(n)alp<3/2}
    \sqrt{N+1} \{\widehat{E}_{N,2} - \abs{\bbE \cos Y(0)}^2\} \overset{d}{\rightarrow} 2 \hspace{0.5mm}\big(\bbE \cos Y(0)\big)  \hspace{0.5mm}\sigma_3 B_3(1),
  \end{equation}
as $N\rightarrow \infty$, where
  \begin{equation}\label{e:sigma3}
    \sigma^2_3 = \sum_{m=2}^{\infty} g_{\cos,m}^2 m! \sum_{k=-\infty}^{\infty} (\gamma_Y(k))^m,
  \end{equation}
and $g_{\cos,m}$ is given by \eqref{e:g_cossin};
  \item [($ii$)] when $3/2<\alpha<2$,
  $$
    \frac{(N+1)^{2-\alpha}}{ L(N+1)} \{\widehat{E}_{N,2} - \abs{\bbE \cos(Y(0))}^2 \}
    $$
    \begin{equation}\label{e:e_2(n)alp>3/2}
    \overset{d}{\rightarrow}
    g_{\cos,2}\beta_{2,\alpha - 1} \widehat{I}_{2}(f_{\alpha - 1,2,1})
    + g_{\sin,1}^2 \beta_{1,\alpha/2}^2 \widehat{I}_{1}^2(f_{\alpha/2,1,1}),
  \end{equation}
as $N\rightarrow \infty$, where $g_{\cos,2}$ and $g_{\sin,1}$ are given by \eqref{e:g_cossin}, $\beta_{2,\alpha - 1}$ and $\beta_{1,\alpha/2}$ are defined in \eqref{e:beta_kH}, $f_{H,k,t}$ is the kernel function defined in \eqref{e:f_Hkt}, $\widehat{I}_{k}(f)$ is defined in \eqref{e:widehatI_k(f)} and $L(N+1)$ is defined in \eqref{e:L_2(k)fGn}.
\end{itemize}
\end{proposition}

\begin{remark}
Theorem \ref{t:E(n)dist} does not cover the instance $\alpha = 3/2$. The borderline case between weak and strong superdiffusion is usually not encompassed by central or non-central limit theorems and requires separate efforts. In mathematically similar circumstances to those considered in Theorem \ref{t:E(n)dist}, the asymptotic distribution can be either Gaussian or non-Gaussian, usually with a nonstandard convergence rate (see, for instance, \cite{guyon:leon:1989}, \textit{Remarque} II-3.3; \cite{buchmann:chan:2009}, Corollary 2.1; \cite{zhang:crizer:schoenfisch:hill:didier:2018}, Table 1).
\end{remark}

We illustrate the claims in Theorem \ref{t:E(n)dist} by means of a Monte Carlo study. Figure \ref{fig:histfitE(n)} shows histograms of the finite sample distribution of $\widehat{E}_N(n)$, where the red normal curve uses the sample mean and sample variance. The three plots on the left-hand side display histograms for ${\mathfrak R}\widehat{E}_N(n)$, whereas the plots on the right-hand side are for ${\mathfrak I}\widehat{E}_N(n)$. In the simulation, the value of $\alpha$ for the top, middle and bottom plots are 0.6, 1, 1.8 respectively. The parameter $n$ is set to 30, and the path length is $2^{10} = 1024$. As expected, all the plots indicate a Gaussian distribution, with one exception, namely, the plot for ${\mathfrak R}\widehat{E}_N(n)$ when $\alpha = 1.8 > 3/2$. The latter is right-skewed as a consequence of the fact that the distribution is a linear combination of a Rosenblatt-type and a chi-squared distribution, both of which are, indeed, right-skewed (see \cite{veillette:taqqu:2013} on the former distribution).
\begin{figure}[htbp]
\begin{center}
\includegraphics[scale=0.4]{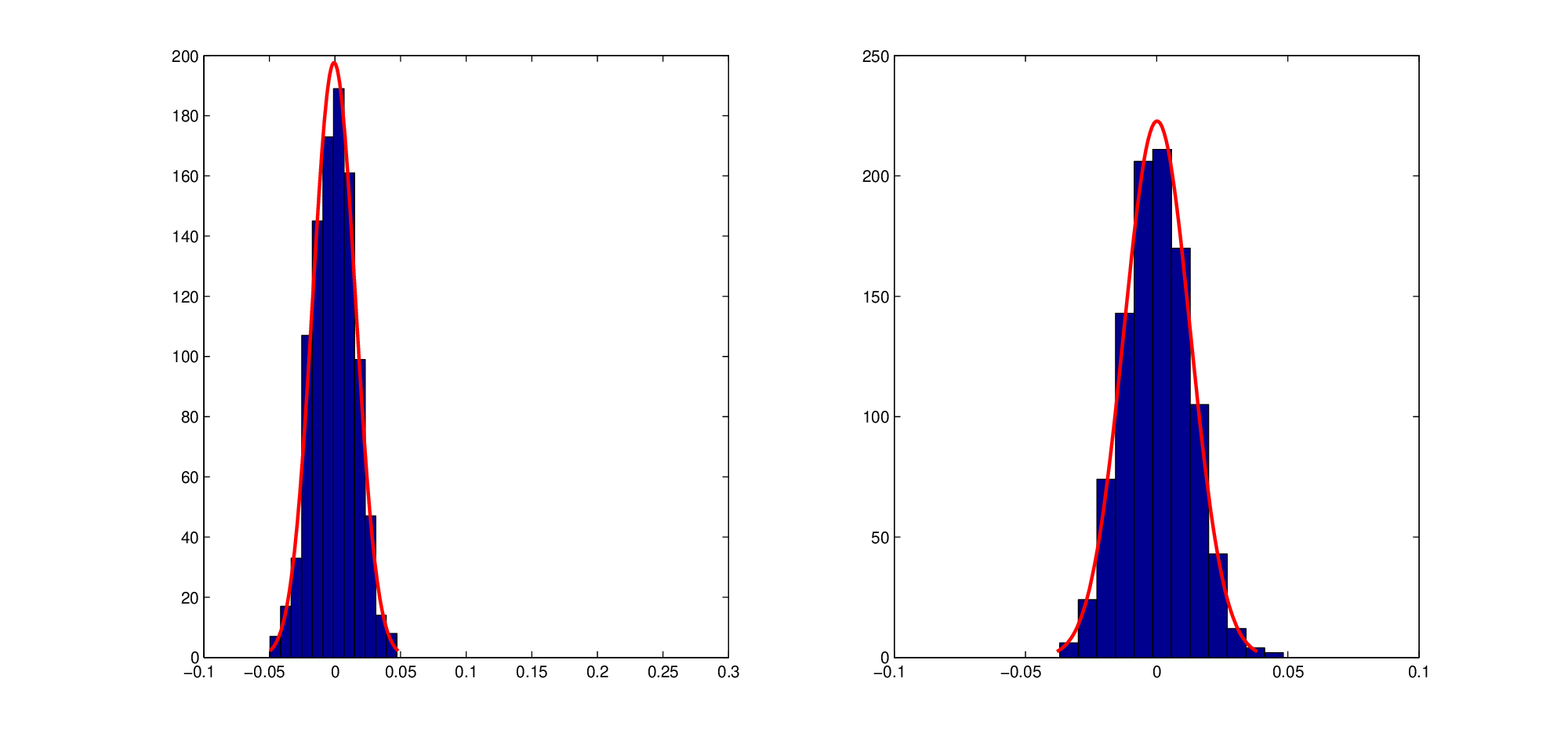}\\
\includegraphics[scale=0.4]{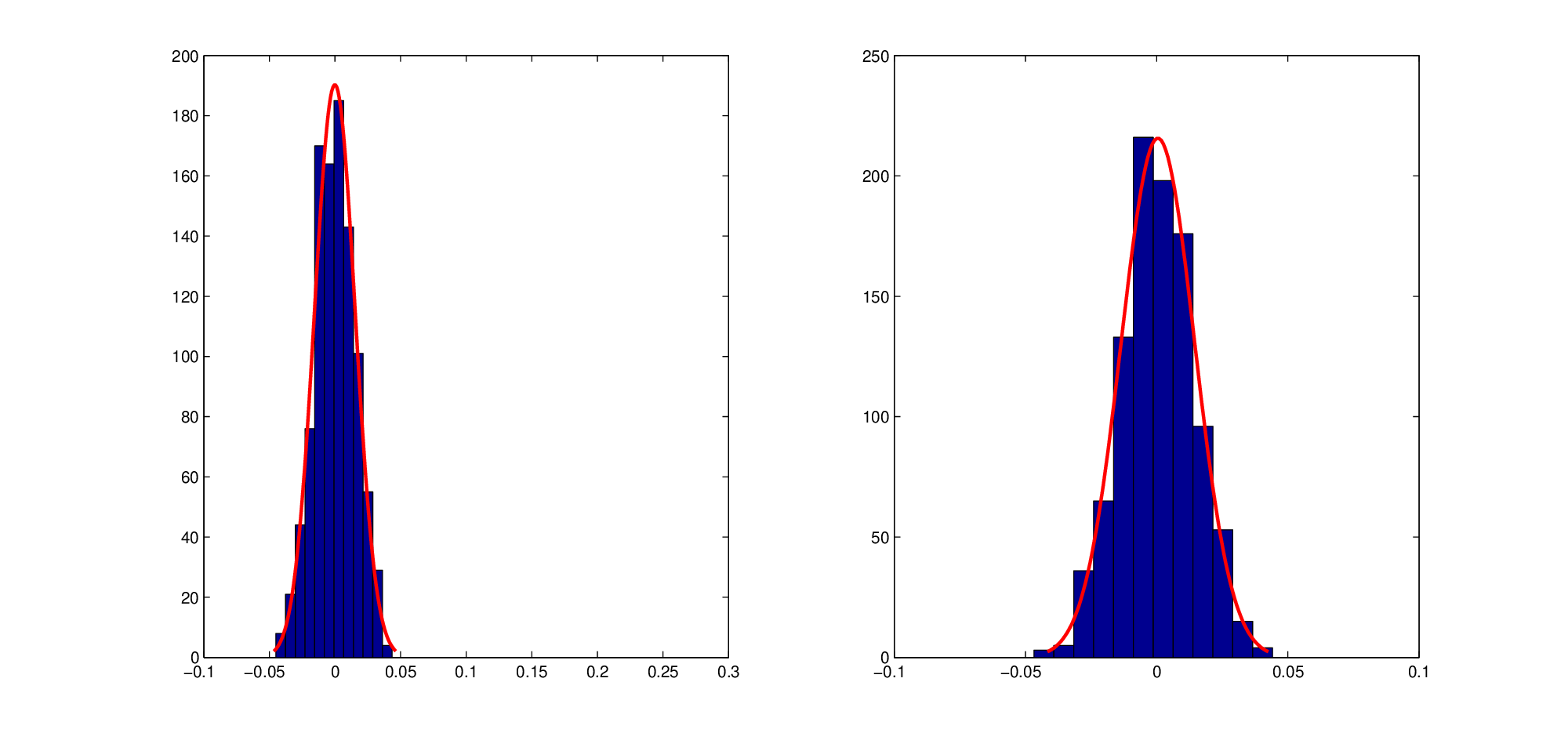}\\
\includegraphics[scale=0.4]{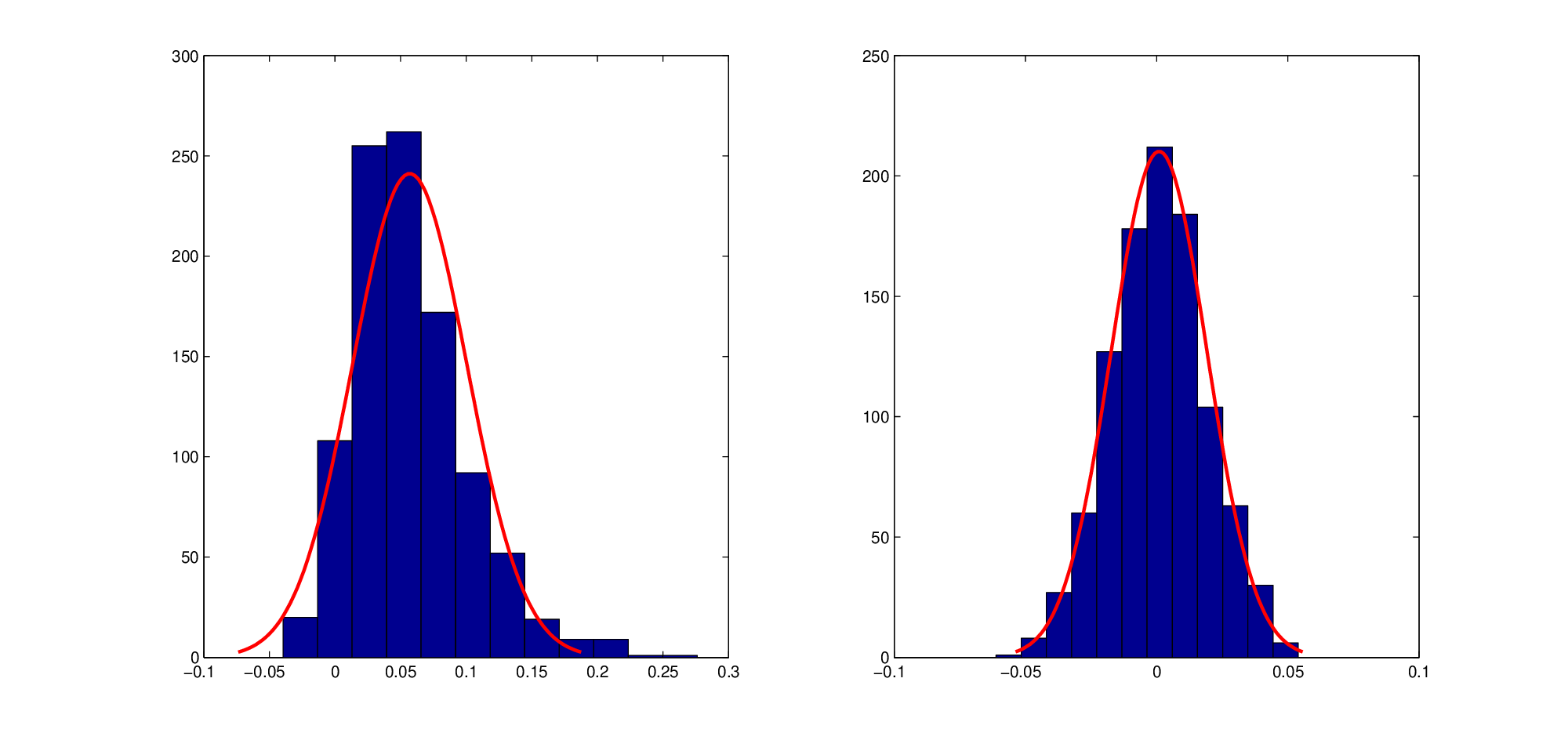}
\caption{Histogram of $\widehat{E}_N(n)$ with fitted normal curve. Left: ${\mathfrak R}\widehat{E}_N(n)$. Right: ${\mathfrak I}\widehat{E}_N(n)$. Top: $\alpha = 0.6$. Middle: $\alpha = 1.0$. Bottom: $\alpha = 1.8$.}
\label{fig:histfitE(n)}
\end{center}
\end{figure}

Figure \ref{fig:conv_rateE(n)} depicts a Monte Carlo study of the convergence rate of $\widehat{E}_N(n)$. In the simulations, $n=30$, $\alpha = 0.6, 1, 1.8$, $N = 2^{9}, 2^{10}, 2^{11},2^{12},2^{13},2^{14}$. As expected, in the right plot, the convergence rate of $\mathfrak{I}\widehat{E}_N(n)$ for different values of $\alpha$ is the same. In the left plot, the estimated slope for $\alpha = 1.8$ is greater than that for $\alpha = 1$ and $\alpha = 0.6$, which in turn are equal. This illustrates the fact that $\mathfrak{R}\widehat{E}_N(n)$ has a nonstandard convergence rate when $\alpha = 1.8$.
\begin{figure}[htbp]
\begin{center}
\includegraphics[scale=0.5]{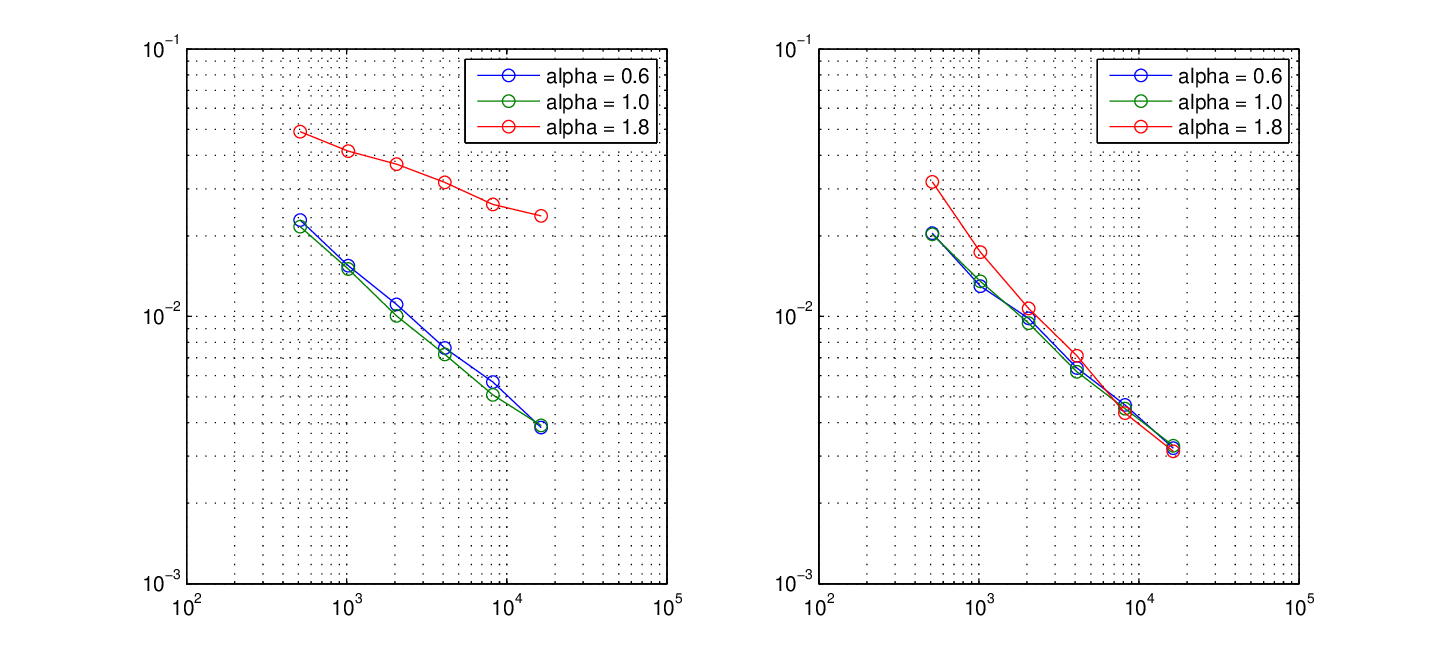}
\caption{$x$-axis: $\log (N+1)$; $y$-axis: $\log$ of the sample standard deviation. Left: ${\mathfrak R}\widehat{E}_N(n)$. Right: ${\mathfrak I}\widehat{E}_N(n)$.}
\label{fig:conv_rateE(n)}
\end{center}
\end{figure}

\begin{remark}\label{r:multiple_n}
Although Theorem \ref{t:E(n)dist} provides the asymptotic distribution of $\widehat{E}_N(n)$ for a single value of $n$, similar techniques can be used to develop the asymptotic distribution of the random vector
\begin{equation}\label{e:vector_EN(n)}
\big(\widehat{E}_N(n_1),\hdots,\widehat{E}_N(n_j)\big), \quad  N \rightarrow \infty,
\end{equation}
for $n_1,\hdots,n_j \in \bbN$.
\end{remark}

\subsection{Discussion: consequences for physical modeling}\label{s:discussion}



As mentioned in the introduction, Theorem \ref{t:E(n)dist} allows us to quantify the uncertainty involved in the use of the estimator $\widehat{E}_N(n)$, and also paves the way for robust mixing detection procedures based on a single observed path $\{Y(k)\}_{k=0,1\hdots,N}$ (cf.\ \cite{magdziarz:weron:2011}, Section IV). In other words, starting from Theorem \ref{t:E(n)dist}, it is possible to develop many mixing detection methods with quantifiable or estimable error margins. For the sake of illustration, we provide a discussion.

In the parametric range $0 < \alpha < 3/2$, this can be done by means of Gaussian-based confidence intervals or hypothesis tests (for a summary of the following procedure in the form of pseudocode, see Section \ref{s:pseudocode}). We can frame the null hypothesis of mixing detection as
\begin{equation}\label{e:H0}
H_0: \textnormal{the anomalous diffusion is mixing}.
\end{equation}
To fix ideas, under the assumptions of Theorem \ref{t:E(n)dist}, suppose for the moment that the following parameters are known, namely,
\begin{equation}\label{e:Re_part_centering}
\mu_{{\mathfrak R}} :=  \bbE \cos\bigg(\frac{\widetilde{Z}_n(0)}{\sigma(n)}\bigg) - \abs{\bbE \cos Y(0)}^2 = e^{-(1-\gamma_Y(n))}- e^{-1} 
\end{equation}
and
\begin{equation}\label{e:asympt_variances}
\theta^2_{{\mathfrak R}} := \Var\big[ \sigma_{1,n} B_1(1) + 2 \hspace{0.5mm}\big(\bbE \cos Y(0)\big) \hspace{1mm}\sigma_3 B_3(1) \big], \quad \theta^2_{{\mathfrak I}} := \Var\big[\sigma_{2,n} B_2(1)\big].
\end{equation}
They correspond, respectively, to the centering term for ${\mathfrak R}\widehat{E}_N(n)$ and to the asymptotic variances of ${\mathfrak R}\widehat{E}_N(n)$ and ${\mathfrak I}\widehat{E}_N(n)$. Further recall that $\bbE \mathfrak{I} \widehat{E}_N(n) = 0$ and that $B_2(1)$ is uncorrelated with $B_\ell(1)$, $\ell =1,3$. By Theorem \ref{t:E(n)dist},
\begin{equation}\label{e:chi-2_type_statistic}
{\mathcal W}_N:=\Big\{\frac{\sqrt{N+1} \big(\hspace{0.5mm}\mathfrak{R} \widehat{E}_N(n)-\mu_{{\mathfrak R}}\big)}{\theta_{{\mathfrak R}}} \Big\}^2+ \Big\{\frac{\sqrt{N+1}\hspace{0.5mm} \mathfrak{I} \widehat{E}_N(n)}{\theta_{{\mathfrak I}}}\Big\}^2 \stackrel{d}\rightarrow \chi^2_2, \quad N \rightarrow \infty.
\end{equation}
The limit in distribution \eqref{e:chi-2_type_statistic} can now be used to test \eqref{e:H0} in a Fisherian style \cite{christensen:2005}. In other words, at a preset significance level $\eta \in [0,1]$, we define the rejection region of $H_0$ as
\begin{equation}\label{e:R_eta}
R_{\eta} = \{w: w > \chi^2_{2}(\eta)\} \subseteq \bbR.
\end{equation}
In \eqref{e:R_eta}, $\chi^2_{2}(\eta)$ is the scalar such that $\bbP(W > \chi^2_{2}(\eta)) = \eta$ for a random variable $W \sim \chi^2_2$. The limit in distribution \eqref{e:chi-2_type_statistic} ensures that, for large $N$ and under the assumptions of Theorem \ref{t:E(n)dist}, the conditional probability
$$
\bbP( {\mathcal W}_N  \in R_{\eta} | \textnormal{$H_0$ is true})
$$
is close to $\eta$.

Since the parameters $\mu_{{\mathfrak R}}$, $\theta^2_{{\mathfrak R}}$ and $\theta^2_{{\mathfrak I}}$ are generally unknown, in practice they need to be estimated. Theorem \ref{t:E(n)dist} -- in particular, expressions \eqref{e:corrB1B3} and \eqref{e:sigma_12} -- shows that, for this purpose, one needs a better understanding of the correlation structure of $Y$ and $\widetilde{Z}_n$. This can be obtained by first estimating $\alpha$ by means of the $\MSD$, for example, as commonly done in the biophysical literature (e.g., \cite{valentine:kaplan:thota:crocker:gisler:prudhomme:beck:weitz:2001,lieleg:vladescu:ribbeck:2010,zhang:crizer:schoenfisch:hill:didier:2018}). So, based on an estimator $\widehat{\alpha}$, we can obtain an approximation $\widehat{\gamma}_{Y}(n)$ (e.g., by further assuming $\gamma_Y(n)$ is close to the covariance of a fGn with parameter $\widehat{H} = \widehat{\alpha}/2$). This way, we arrive at the estimator
\begin{equation}\label{e:Re_part_centering_estimator}
\widehat{\mu}_{{\mathfrak R}} :=  e^{-(1- \widehat{\gamma}_Y(n))}- e^{-1}.
\end{equation}
Since $\sigma^2(n) = (2[\Var Y(0) - \gamma_{Y}(n)])^{-1}$, then we can further estimate
\begin{equation}\label{e:sigma2-hat}
\widehat{\sigma}^2(n) = (2[\Var Y(0) - \widehat{\gamma}_{Y}(n)])^{-1}.
\end{equation}
So, starting from the relation $\widehat{G}_{2,n}(x) := \sin(x/\widehat{\sigma}(n))$, we can compute $\widehat{g}_{2,n,m} = (1/m!)(\widehat{G}_{2,n},H_m)_{L^2(\bbR,\phi)}$. In turn, in view of \eqref{e:gammaZtilde(k)=O(k^(alpha-3))}, we can approximate
\begin{equation}\label{e:gamma-hat_Z-tilde(k)}
\widehat{\gamma}_{\widetilde{Z}_n}(k) = O(k^{\widehat{\alpha}-3}).
\end{equation}
Hence, since $\Var B_2(1) = 1$, in view of \eqref{e:asympt_variances} we obtain the estimator
\begin{equation}\label{e:theta^2_Im_estimator}
\widehat{\theta^2_{\Im}} =  \widehat{\sigma^2_{2,n}} = \sum_{m=1}^{M_2} \widehat{g}^2_{2,n,m}\hspace{0.5mm} m! \sum_{\ell=-L_2}^{L_2} (\widehat{\gamma}_{\widetilde{Z}_n}(\ell))^m
\end{equation}
for preset (and large enough) values of $M_2$ and $L_2$. On the other hand, based on \eqref{e:sigma2-hat} and on the relation $\widehat{G}_{1,n}(x) := \cos(x/\widehat{\sigma}(n))$, we can compute $\widehat{g}_{1,n,m} = (1/m!)(\widehat{G}_{1,n},H_m)_{L^2(\bbR,\phi)}$. Again based on the estimator $\widehat{\alpha}$, we can compute an approximation $\widehat{\gamma}_{\widetilde{Z}_n,Y}$. So, in view of \eqref{e:gamma-hat_Z-tilde(k)}, for preset (and large enough) values of $L_1$, $M_1$, $L$ and $M$ we can estimate
$$
\widehat{\sigma^2_{1,n}} = \sum_{m=1}^{M_1} \widehat{g}^2_{1,n,m}\hspace{0.5mm} m! \sum_{\ell=-L_1}^{L_1} (\widehat{\gamma}_{\widetilde{Z}_n}(\ell))^m, \quad \widehat{\sigma^2_{3}} = \sum_{m=1}^{M_3} \widehat{g}^2_{\cos,m}\hspace{0.5mm} m! \sum_{\ell=-L_3}^{L_3} (\widehat{\gamma}_{Y}(\ell))^m
$$
and
$$
\widehat{\cov}(B_1(1), B_3(1)) = - \sum^{L}_{\ell=-L}\sum_{m = 2}^{M} \widehat{g}_{1,n,m} \hspace{0.5mm}\widehat{g}_{\cos,m} \hspace{0.5mm}m!  (\widehat{\gamma}_{\widetilde{Z}_n,Y}(\ell))^m.
$$
Therefore, in view of \eqref{e:asympt_variances} and since $\Var B_1(1) = 1 = \Var B_3(1)$, we obtain the estimator
\begin{equation}\label{e:theta^2_Re_estimator}
\widehat{\theta^2_{\Re}} = \widehat{\sigma^2_{1,n}}+ 4 e^{-2} \cdot \widehat{\sigma^2_{3}} + 4 e^{-1}\cdot \sqrt{\widehat{\sigma^2_{1,n}}} \hspace{0.5mm}\widehat{\cov}(B_1(1), B_3(1)).
\end{equation}

Specifically in regard to the parameters $\theta^2_{{\mathfrak R}}$ and $\theta^2_{{\mathfrak I}}$, one alternative to \eqref{e:theta^2_Im_estimator} and \eqref{e:theta^2_Re_estimator} is to use resampling methods (e.g., \cite{lahiri:2013}).

In the strongly superdiffusive regime $3/2 < \alpha < 2$, Theorem \ref{t:E(n)dist} reveals that using traditional, Gaussian-inspired procedures such as \eqref{e:R_eta} potentially leads to significant inaccuracy (cf.\ \cite{zhang:crizer:schoenfisch:hill:didier:2018}, Section 4.1, on related issues involving the $\MSD$). In fact, on one hand, for $3/2 < \alpha < 2$ the term $\mathfrak{R} \widehat{E}_N(n)$ displays the unconventional convergence rate $\frac{(N+1)^{2-\alpha}}{L(N+1)}$ (instead of the standard rate $\sqrt{N+1}$). This leads to arbitrarily large error margins, since, by \eqref{e:weak_limit_Re_E-hat_N(n)_3/2<alpha<2},
$$
\sqrt{N+1} \bigg( \mathfrak{R} \widehat{E}_N(n) - \bbE \cos\bigg(\frac{\widetilde{Z}_n(0)}{\sigma(n)}\bigg) + \abs{\bbE \cos(Y(0))}^2 \bigg)
$$
$$
= \frac{\sqrt{N+1}}{(N+1)^{2-\alpha}/ L(N+1)}\cdot \frac{(N+1)^{2-\alpha}}{L(N+1)} \bigg( \mathfrak{R} \widehat{E}_N(n) - \bbE \cos\bigg(\frac{\widetilde{Z}_n(0)}{\sigma(n)}\bigg) + \abs{\bbE \cos(Y(0))}^2 \bigg)
$$
$$
\stackrel{\bbP}\sim \infty \cdot \Big( g_{\cos,2}\beta_{2,\alpha - 1} \widehat{I}_{2}(f_{\alpha - 1,2,1})+ g_{\sin,1}^2 \beta_{1,\alpha/2}^2 \widehat{I}_{1}^2(f_{\alpha/2,1,1}) \Big), \quad N \rightarrow \infty.
$$
On the other hand, the asymptotic distribution of $\mathfrak{R} \widehat{E}_N(n)$ is of Hermite type. In fact, partly due to their complexity, the use of Hermite-type distributions in modeling is still an active topic of research in the Probability and Statistics literature (e.g., \cite{bardet:tudor:2014,bai:taqqu:2018}).

In general, an interesting and natural idea is to start from the random vector \eqref{e:vector_EN(n)}, based on multiple values $n_1,\hdots,n_j$, as to arrive at a testing procedure with greater power -- namely, the ability to reject the null hypothesis \eqref{e:H0} when it is false. Nevertheless, this further involves carefully handling the correlations appearing in the asymptotic distribution of \eqref{e:vector_EN(n)}.

\section{Conclusion and open problems}\label{s:conclusion_and_open_problems}
The statistic $\widehat{E}_N(n)$ is proposed in \cite{magdziarz:weron:2011} for mixing detection in anomalous diffusion. In this paper, we establish the asymptotic distribution of $\widehat{E}_N(n)$ for fixed $n$ as $N \rightarrow \infty$. We assume the underlying stochastic process $Y$ is a fractional Gaussian stationary sequence and a single sample path is available. We show that ${\mathfrak I} \widehat{E}_N(n)$ always converges to a Gaussian distribution at a standard rate of convergence. On the other hand, depending on the anomalous regime ($\alpha$ parameter range), ${\mathfrak R} \widehat{E}_N(n)$ converges to a Gaussian or non-Gaussian distribution at standard or nonstandard convergence rates.

This work leads to a number of open problems and research directions. These include
\begin{itemize}
\item [$(i)$] as discussed in Section \ref{s:discussion}, investigating the finite sample performance of specific mixing detection procedures based on $\widehat{E}_N(n)$ and with controllable error margins;
\item [$(ii)$] studying the asymptotic behavior of $\widehat{E}_N(n)$ in the double limit $n,N \rightarrow \infty$, for increased accuracy;
\item [$(iii)$] understanding the asymptotic behavior of $\widehat{E}_N(n)$ or related statistics starting from mathematically different classes of fractional processes such as non-Gaussian infinitely divisible processes or continuous time random walks. This would allow for the quantification of the uncertainty involved in testing the null hypothesis \eqref{e:H0} starting from these alternative classes of anomalous diffusion models. In particular, this naturally leads to the question of how robust the use of $\widehat{E}_N(n)$ is, in terms of error margins, over broad classes of Gaussian and non-Gaussian anomalous diffusion models;
\item [$(iv)$] establishing the asymptotic behavior of ergodicity detection statistics such as the one proposed in \cite{magdziarz:weron:2011}.
\end{itemize}

\section{Acknowledgments}

The authors are grateful to an anonymous reviewer for the constructive comments that led to a significantly improved manuscript.

\appendix

\setcounter{equation}{0}

\section{Central and non-central limit theorems}\label{s:central_non-central}

\renewcommand{\theequation}{\thesection.\arabic{equation}}


For the reader's convenience, in this section we provide basic results on moments and central or non-central limit theorems for Gaussian stationary processes. Proposition \ref{prop:g1g2cov} and Theorems \ref{thm:noncentral}, \ref{thm:central}, \ref{thm:central_multi} and \ref{thm:noncentral_multi} correspond to Proposition 5.1.4 and Theorems 5.3.1, 5.4.1, 5.7.1 and 5.7.4, respectively, in \cite{pipiras:taqqu:2017}. Hereinafter, $\overset{\textnormal{f.d.d.}}{\rightarrow}$ denotes the convergence of finite-dimensional distributions.

Let $(X,Y)^{\top}$ be a Gaussian random vector. The following proposition allows us to calculate the covariance between $G_1(X)$ and $G_2(Y)$, where $G_1$ and $G_2$ are suitable transformations.
\begin{proposition}\label{prop:g1g2cov}
Let $(X,Y)^{\top}$ be a Gaussian vector with $\bbE X = \bbE Y = 0$ and $\bbE X^2 = \bbE Y^2 = 1$. Suppose $G_1,G_2 \in L^2(\bbR, \phi)$ and let $g_{1,n}$ and $g_{2,n}$, $n\geq 0$, be the coefficients in the Hermite expansions of $G_1$ and $G_2$, respectively, as in \eqref{e:g_k}. Then,
$$
\bbE G_1(X)G_2(Y) = \sum_{m=0}^{\infty} g_{1,m}g_{2,m}m!(\bbE XY)^m
$$
and
$$
\cov (G_1(X),G_2(Y)) = \sum_{m=1}^{\infty} g_{1,m}g_{2,m}m!(\bbE XY)^m.
$$
\end{proposition}

The following theorem establishes that the (weak) limit of \eqref{e:V_N(t)} for $R = 1$ is an Hermite process. Recall that the Hermite process is only Gaussian when the Hermite rank of $G$ is $k=1$.
\begin{theorem}\label{thm:noncentral}
 Let $\{X_n\}_{n \in \bbN}$ be a centered stationary Gaussian sequence with autocovariance function $\gamma_X$ as in \eqref{e:gamma_X} and such that $\bbE X^2_1 = 1$. Let $G$ be a function of Hermite rank $k\geq 1$ and suppose that $d$ as in \eqref{e:gamma_X} satisfies
  $$
  d\in \bigg( \frac{1}{2}\bigg( 1 - \frac{1}{k} \bigg), \frac{1}{2} \bigg).
  $$
  Then,
\begin{equation}\label{e:noncentral_ffd}
  \frac{1}{(L(N))^{k/2} N^{k(d-1/2)+1}} \sum_{n=1}^{[Nt]} \big( G(X_n) - \bbE G(X_n) \big) \overset{\textnormal{f.d.d.}}{\rightarrow} g_k \beta_{k,H} Z_H^{(k)}(t), \quad t\geq 0,
\end{equation}
where $g_k$ is the first non-zero coefficient in \eqref{e:g_k} and
\begin{equation}\label{e:beta_kH}
  \beta_{k,H} = \bigg( \frac{k!}{H(2H-1)} \bigg)^{1/2}.
\end{equation}
In \eqref{e:noncentral_ffd}, the self-similarity parameter is given by
$$
H = k\bigg(d - \frac{1}{2}\bigg) + 1 \in \bigg(\frac{1}{2},1\bigg)
$$
and $\{Z_H^{(k)}(t)\}_{t\in \bbR}$ is the Hermite process defined by \eqref{e:hermiteprocess}.
\end{theorem}
The following result shows that, in some cases, the limit of partial sum processes as in \eqref{e:V_N(t)} (for $R = 1$) is the usual Brownian motion.
\begin{theorem}\label{thm:central}
  Let $\{X_n\}_{n \in \bbN}$ be a centered stationary Gaussian sequence with autocovariance function $\gamma_X$ and such that $\bbE X^2_1 = 1$. Let $G$ be a function with Hermite rank $k\geq 1$ in the sense of Definition \ref{def:hermiterank}. If
\begin{equation}\label{e:sum_gammax_conv}
  \sum_{\ell=1}^{\infty} \abs{\gamma_X(\ell)}^k < \infty,
\end{equation}
  then
  \begin{equation}\label{e:brownian_limit}
    \frac{1}{N^{1/2}} \sum_{n=1}^{[Nt]} \big( G(X_n) - \bbE G(X_n) \big) \overset{\textnormal{f.d.d.}}{\rightarrow} \sigma B(t), \quad t\geq 0,
  \end{equation}
  where $\{B(t)\}_{t\geq 0}$ is a standard Brownian motion and
  $$
  \sigma^2 = \sum_{m=k}^{\infty} g_m^2 m! \sum_{\ell=-\infty}^{\infty} (\gamma_X(\ell))^m.
  $$
In particular, if
$$
d\in \bigg(0, \frac{1}{2}\bigg(1 - \frac{1}{k}\bigg)\bigg)
$$
in expression \eqref{e:gamma_X} for the autocovariance function  $\gamma_X(\cdot)$ of $X$, then the convergence (\ref{e:brownian_limit}) to a Brownian motion holds.
\end{theorem}

\begin{example}
If $k=1$ in \eqref{e:sum_gammax_conv}, then the absolute summability of the autocovariance function leads to an ordinary Brownian limit. If $k=2$, on the other hand, this limit only emerges when $d<\frac{1}{4}$.
\end{example}

We now turn to multivariate limit theorems. Consider the vector-valued random process \eqref{e:V_N(t)}. The following theorem provides sufficient conditions for the process \eqref{e:V_N(t)} to converge, as $N\rightarrow +\infty$, to a multivariate Gaussian process with dependent Brownian motion marginals.
\begin{theorem}\label{thm:central_multi}
  Let $\{X_n\}_{n \in \bbN}$ be a centered stationary Gaussian sequence with autocovariance function $\gamma_X$ and such that $\bbE X^2_1 = 1$. Let $G_r, r = 1,\hdots,R,$ be deterministic functions with respective Hermite ranks $k_r\geq 1$ , $r=1,\hdots,R$. If
  $$
  \sum_{n=1}^{\infty} \abs{\gamma_X(n)}^{k_r} < \infty,\quad r=1,\hdots,R,
  $$
  then
  $$
  {\mathbf V}_N(t) \overset{\textnormal{f.d.d.}}{\rightarrow} {\mathbf V} (t),\quad t\geq 0,
  $$
where ${\mathbf V}_N(t)$ is given by \eqref{e:V_N(t)} with $A_r(N) = N^{1/2},r=1,\hdots,R.$ The limit process can be expressed as
\begin{equation}\label{e:v(t)_multicentral}
{\mathbf V}(t) = \big( \sigma_1 B_1(t),\hdots,\sigma_R B_R(t) \big)^{\top},
\end{equation}
where
  $$
  \sigma_r^2 = \sum_{m=k_r}^{\infty} g_{r,m}^2 m! \sum_{\ell=-\infty}^{\infty} (\gamma_X(\ell))^m,\quad r=1,\hdots,R,
  $$
and $g_{r,m}$, $r=1,\hdots,R,$ are the coefficients in the Hermite expansion \eqref{e:g_k} of $G_r$. In \eqref{e:v(t)_multicentral}, $\{B_r(t)\}_{t\in\bbR} , r=1,\hdots,R,$ are standard Brownian motions with cross-covariance
  $$
  \bbE B_{r_1}(t_1) B_{r_2}(t_2) = (t_1 \wedge t_2) \frac{\sigma_{r_1,r_2}}{\sigma_{r_1} \sigma_{r_2}}, \quad t_1,t_2\geq 0,
  $$
  and
  $$
  \sigma_{r_1,r_2} = \sum_{m=k_{r_1} \vee k_{r_2}}^{\infty} g_{r_1,m} g_{r_2,m} m! \sum_{n= - \infty}^{\infty} \gamma_X(n)^m
  $$
($a \wedge b = \min\{a,b\}$ and $a \vee b = \max\{a,b\}$).
\end{theorem}
The next result concerns the general case where the resulting limit law for \eqref{e:V_N(t)} is a multivariate process with dependent Hermite processes as marginals. The weak limit involves stochastic integrals of the form \eqref{e:widehatI_k(f)}.
\begin{theorem}\label{thm:noncentral_multi}
  Let $\{X_n\}_{n \in \bbN}$ be a centered stationary Gaussian sequence with autocovariance function $\gamma_X$ as in \eqref{e:gamma_X} and such that $\bbE X^2_1 = 1$. Let $G_r,r=1,\hdots,R,$ be deterministic functions with respective Hermite ranks $k_r \geq 1$, $r=1,\hdots,R$. For $d$ as in \eqref{e:gamma_X}, suppose
  $$
  d\in \bigg(\frac{1}{2}\bigg(1-\frac{1}{k_r}\bigg), \frac{1}{2}\bigg),\quad r=1,\hdots,R.
  $$
Consider the process ${\mathbf V}_N(t)$ given by \eqref{e:V_N(t)} with
$$
A_r(N) = (L(N))^{k_r/2} N^{k_r (d - 1/2) + 1}, \quad r = 1,\hdots,R.
$$
Then,
  \begin{equation}\label{e:VN(t)->V^(d)(t)}
\bbR^{R} \ni  {\mathbf V}_N(t) \overset{\textnormal{f.d.d.}}{\rightarrow} {\mathbf V}^{(d)} (t), \quad t\geq 0.
  \end{equation}
In \eqref{e:VN(t)->V^(d)(t)}, the limit process can be represented as
$$
{\mathbf V}^{(d)} (t) \overset{\textnormal{f.d.d.}}{=} \Big\{g_{r,k_r} \beta_{k_r,H_r} \widehat{I}_{k_r}(f_{H_r,k_r,t})\Big\}_{r=1,\hdots,R},
$$
where
  $$
  H_r = k_r\bigg(d - \frac{1}{2}\bigg) + 1 \in \bigg(\frac{1}{2},1\bigg),\quad r=1,\hdots,R,
  $$
$g_{r,k_r}$ is the first non-zero coefficient in the Hermite expansion of $G_r$ in \eqref{e:g_k}, $\beta_{k,H}$ is the constant given in \eqref{e:beta_kH}, and $f_{H,k,t}$ is the kernel function defined by \eqref{e:f_Hkt}.
\end{theorem}

\section{Proofs}\label{app:mixingproofs}

The following two lemmas are mentioned in the Introduction. They establish that the assumptions of Theorem \ref{t:E(n)dist} hold for fGn and the fOU process.
\begin{lemma}\label{lem:gammaZ}
Let $Y$ be a fGn or a fOU process. For $n \in \bbN$, let $\{\gamma_{\widetilde{Z}_n}(k)\}_{k\in\bbZ}$ be the autocovariance function of the associated sequence $\{\widetilde{Z}_n(k)\}_{k\in \bbZ}$ as in \eqref{e:defZn}. Then,
\begin{equation}\label{e:sum_gamma_Z(n)}
  \sum_{k=1}^{\infty} |\gamma_{\widetilde{Z}_n}(k)| < \infty, \quad \sum_{k=1}^{\infty} |\gamma_{\widetilde{Z}_n}(k)|^2 < \infty.
\end{equation}
\end{lemma}
\begin{proof}
First suppose $Y$ is a fGn. We start by developing the autocovariance function of $\{Y(k)\}_{k\in \bbZ}$. By a Taylor expansion of order 2,
 $$
 \gamma_Y(k) = \frac{1 }{2} \abs{k}^{\alpha}\bigg(\abs{1-\frac{1}{k}}^{\alpha} - 2 + \abs{1+\frac{1}{k}}^{\alpha}\bigg)
 $$
 $$
 = \frac{1 }{2} \abs{k}^{\alpha} \bigg( 1 - \alpha k^{-1} + \frac{\alpha(\alpha - 1)}{2} k^{-2}  + O(k^{-3}) - 2 + 1 + \alpha k^{-1} + \frac{\alpha(\alpha - 1)}{2} k^{-2}  + O(k^{-3})\bigg)
 $$
\begin{equation}\label{e:gamma_y(k)}
  = \frac{ \alpha(\alpha - 1) }{2}k^{\alpha - 2} + O(k^{\alpha -3}),
\end{equation}
as $k\rightarrow\infty$. Consider the standardized increments of fGn, $\widetilde{Z}_n(k) = \sigma(n)[Y(n+k) - Y(k)]$. By \eqref{e:gamma_y(k)} and a Taylor expansion of order 1, for $k\in\bbZ$,
 $$
 \gamma_{\widetilde{Z}_n}(k) = \sigma^2(n) \bbE (Y(n+k+j) - Y(k+j))(Y(n+j) - Y(j))
  $$
  $$
  = -\sigma^2(n)[\gamma_Y(k+n) -2\gamma_Y(k) + \gamma_Y(k-n) ]
 $$
 $$
 = -\frac{\sigma^2(n) \alpha(\alpha - 1) }{2}(\abs{k+n}^{\alpha -2} - 2\abs{k}^{\alpha - 2} + \abs{k-n}^{\alpha-2})+ O(k^{\alpha -3})
 $$
 $$
 = \frac{\sigma^2(n) \alpha(\alpha - 1) }{2}\abs{k}^{\alpha - 2}\bigg(2 - \abs{1-\frac{n}{k}}^{\alpha-2} - \abs{1+\frac{n}{k}}^{\alpha -2}\bigg)+ O(k^{\alpha -3})
 $$
 \begin{equation}\label{e:gammaZtilde=O(k^(alpha-3))_fGn}
 = \frac{\sigma^2(n) \alpha(\alpha - 1) }{2}\abs{k}^{\alpha - 2}O(k^{-1}) + O(k^{\alpha -3}) = O(k^{\alpha -3}),
 \end{equation}
as $k\rightarrow\infty$. Therefore, \eqref{e:sum_gamma_Z(n)} holds.

Now suppose $Y$ is a fOU process, $H \neq 1/2$. By Cheridito et al.\ \cite{cheridito:kawaguchi:maejima:2003}, Theorem 2.3, for any $k,L \in \bbN$,
$$
\gamma_Y(k) = \frac{\sigma^2}{2} \sum^{L}_{\ell =1} \lambda^{-2 \ell} \Big( \prod^{2\ell-1}_{q=0}(2H-q)\Big) q^{2H-2\ell} + O(k^{2H-2L-2})
$$
\begin{equation}\label{e:gamma_y(k)_fOU}
= \frac{\sigma^2 2H(2H-1) }{2\lambda^{2}}  k^{2H-2} + O(k^{2H-3}).
\end{equation}
Moreover, for some function $\sigma(n)$, again by a Taylor expansion of order 1,
$$
\gamma_{\widetilde{Z}_n}(k) = - \sigma^2(n) \{ \gamma_Y(k+n) - 2 \gamma_Y(k) + \gamma_Y(k-n)\}
$$
$$
= - \sigma^2(n) \frac{\sigma^22H(2H-1) }{2\lambda^{2}}   \{ |k+n|^{2H-2} - 2 |k|^{2H-2} + |k-n|^{2H-2}
$$
$$
+ O(|k+n|^{2H-3}) + O(|k|^{2H-3}) + O(|k-n|^{2H-3})\}
$$
\begin{equation}\label{e:gammaZtilde=O(k^(alpha-3))_fOU}
= - \sigma^2(n) \frac{\sigma^22H(2H-1) }{2\lambda^{2}}   |k|^{2H - 2} O(k^{-1}) + O(k^{2H-3}).
\end{equation}
Therefore, \eqref{e:sum_gamma_Z(n)} also holds when $H \neq 1/2$. The case where $H = 1/2$ (the traditional OU process) is straightforward.
\end{proof}

\begin{lemma}\label{l:summability_cross-cov}
Let $Y$ be a fGn or a fOU process with parameter $0 < \alpha < 3/2$ ($\alpha = 2H$). Given expression \eqref{e:defZn}, fix $n \in \bbN$ large enough so that
\begin{equation}\label{e:sigma(n)<1}
\sigma(n) < 1.
\end{equation}
Let $\{\gamma_{\widetilde{Z}_n,Y}(k)\}_{k\in\bbZ}$ be the cross-covariance function of $\{\widetilde{Z}_n(k)\}_{k \in \bbZ}$ and $Y$. Then,
\begin{equation}\label{e:summability_cross-cov}
\sum^{\infty}_{\ell = - \infty} \Big| \sum^{\infty}_{m=2} g_{1,n,m}\hspace{0.5mm}g_{\cos,m}\cdot m! \big( \gamma_{\widetilde{Z}_n,Y}(\ell)\big)^m\Big| < \infty.
\end{equation}
\end{lemma}
\begin{proof}
First suppose $Y$ is a fGn. We start off with the range $1 < \alpha < 3/2$. By Taqqu \cite{taqqu:2003}, Proposition 3.1, $(e)$,
\begin{equation}\label{e:gamma_Y>0}
\gamma_Y(k) > 0, \quad k \in \bbZ.
\end{equation}
Moreover, note that the function $f(x) = x^{\alpha-1}$, $x > 0$, is strictly concave. Then, for $k \in [1,\infty)$,
$$
\frac{1}{2} (k+1)^{\alpha-1}+\frac{1}{2} (k-1)^{\alpha-1} < \Big[ \frac{1}{2} (k+1)+\frac{1}{2} (k-1)\Big]^{\alpha-1} = k^{\alpha-1}.
$$
Equivalently, $\frac{\alpha}{2}\{(k+1)^{\alpha-1} - 2 k^{\alpha-1}+ (k-1)^{\alpha-1}\} < 0$. Since, in addition, $\gamma_Y(k) < 1$ for $k \neq 0$,
\begin{equation}\label{e:mono_decreas}
\textnormal{$\gamma_Y$ is monotonic decreasing on $\bbN \cup \{0\}$.}
\end{equation}
So, for $k \in \bbZ$, recast
\begin{equation}\label{e:gamma-cross_reexpressed}
\gamma_{\widetilde{Z}_n,Y}(k) = \sigma(n) \bbE \big[ \big(Y(n+k)-Y(k)\big)Y(0) \big] = \sigma(n)\big[\gamma_{Y}(n+k) - \gamma_{Y}(k)\big].
\end{equation}
By \eqref{e:sigma(n)<1}, \eqref{e:gamma_Y>0} and \eqref{e:mono_decreas},
\begin{equation}\label{e:sigma(n)|Y-Y|_bound}
0 < \sigma(n)\big|\gamma_{Y}(n+k) - \gamma_{Y}(k)\big| <
\left\{\begin{array}{cc}
\gamma_{Y}(k), & \textnormal{if } \lfloor - \frac{n}{2}\rfloor  < k ;\\
\gamma_{Y}(n+k), & \textnormal{if } k \leq \lfloor - \frac{n}{2} \rfloor .
\end{array}\right.
\end{equation}
So, fix $m \in \bbN \backslash \{1\}$. Then, by \eqref{e:gamma-cross_reexpressed} and \eqref{e:sigma(n)|Y-Y|_bound},
$$
\sum^{\infty}_{k = - \infty} \big| \gamma_{\widetilde{Z}_n,Y}(k)\big|^m
=\Big\{\sum^{\infty}_{k = \lfloor - n/2 \rfloor}+ \sum^{\lfloor - n/2 \rfloor -1}_{k=-\infty}  \Big\}\big| \gamma_{\widetilde{Z}_n,Y}(k)\big|^m
$$
$$
\leq \sum^{\infty}_{k = \lfloor - n/2 \rfloor} \big(\gamma_{Y}(k)\big)^m+ \sum^{\lfloor - n/2 \rfloor-1}_{k=-\infty} \big(\gamma_{Y}(n+k)\big)^m
\leq 2 \sum^{\infty}_{k = - \infty} \big(\gamma_{Y}(k)\big)^m
$$
\begin{equation}\label{e:sum_|gamma_Zn,Y|^m<infty}
= 2 \sum^{\infty}_{k = - \infty} \big|\gamma_{Y}(k)\big|^m < \infty, \quad m \in \bbN \backslash \{1\}.
\end{equation}
Recall that, by \eqref{e:sigma3},
\begin{equation}\label{e:sum_g^2_cos.m!_sum_cov-fGn^m<infty}
\sum^{\infty}_{m=2} g^{2}_{\cos,m} \cdot m! \sum^{\infty}_{k = - \infty}\big( \gamma_{Y}(k)\big)^m  = \sigma^2_3 < \infty.
\end{equation}
Moreover, by repeating the proof of Proposition \ref{prop:E_2(n)} with the function $G_{1,n}(x) = \cos(x/\sigma(n))$ in place of $\cos(x)$, we further conclude that
\begin{equation}\label{e:sum_g^2_1nm.m!_sum_cov-fGn^m<infty}
\sum^{\infty}_{m=2} g^{2}_{1,n,m} \cdot m! \sum^{\infty}_{k = - \infty}\big( \gamma_{Y}(k)\big)^m   < \infty.
\end{equation}
In view of \eqref{e:sum_g^2_cos.m!_sum_cov-fGn^m<infty} and \eqref{e:sum_g^2_1nm.m!_sum_cov-fGn^m<infty}, the left-hand side of \eqref{e:summability_cross-cov} is bounded by
\begin{equation}\label{e:LHS_summability_cross-cov_bound}
 \sum^{\infty}_{m=2} \Big(\frac{g^2_{1,n,m}+g^2_{\cos,m}}{2}\Big) \cdot m! \sum^{\infty}_{k = - \infty} \big| \gamma_{\widetilde{Z}_n,Y}(k)\big|^m
\leq  \sum^{\infty}_{m=2} \big(g^2_{1,n,m}+g^2_{\cos,m}\big)\cdot m! \sum^{\infty}_{k = - \infty} \big(\gamma_{Y}(k)\big)^m < \infty.
\end{equation}
This establishes \eqref{e:summability_cross-cov} for $1 < \alpha < 3/2$.

Now rewrite $\gamma_{Y,\alpha} \equiv \gamma_{Y}$ to denote the autocovariance function of a standard fGn with parameter $0 < \alpha < 2$ ($\alpha = 2H$). Recall that, for a nonzero constant $C_{\alpha}$,
\begin{equation}\label{e:autocov_fGn_decay}
\gamma_{Y,\alpha}(k) \sim C_{\alpha} k^{\alpha -2}, \quad k \rightarrow \infty
\end{equation}
(e.g., Taqqu \cite{taqqu:2003}, Proposition 3.1, $(f)$). Then, for large enough $k > 0$,
\begin{equation}\label{e:gamma-cross_bounded_|gamma-Y|}
0 < |\gamma_{\widetilde{Z}_n,Y}(k)| = \sigma(n)\big|\gamma_{Y,\alpha}(n+k) - \gamma_{Y,\alpha}(k)\big|
= \sigma(n)|\gamma_{Y,\alpha}(k)|\Big|\frac{\gamma_{Y,\alpha}(n+k)}{\gamma_{Y,\alpha}(k)} - 1\Big| \leq |\gamma_{Y,\alpha}(k)|.
\end{equation}
So, fix $0 < \alpha < 1$ and $1 < \alpha_* < 3/2$. In view of \eqref{e:autocov_fGn_decay} and \eqref{e:gamma-cross_bounded_|gamma-Y|}, there exists $k_0 \in \bbN$ such that
\begin{equation}\label{e:|cross-cov|_bound_based_on_gamma_Y,alpha*}
0 < |\gamma_{\widetilde{Z}_n,Y}(k)| = \sigma(n)\big|\gamma_{Y,\alpha}(n+k) - \gamma_{Y,\alpha}(k)\big| \leq
\left\{\begin{array}{cc}
\gamma_{Y,\alpha_*}(k), & \textnormal{if } k_0+1 \leq |k| ;\\
1, & \textnormal{if } |k|\leq k_0.\\
\end{array}\right.
\end{equation}
In view of \eqref{e:|cross-cov|_bound_based_on_gamma_Y,alpha*}, the left-hand side of \eqref{e:summability_cross-cov} is bounded by
$$
 \sum^{\infty}_{m=2} \Big(\frac{g^2_{1,n,m}+g^2_{\cos,m}}{2}\Big) \cdot m! \Big\{\sum_{|k|\leq k_0} + \sum_{k_0+1 \leq |k| } \Big\}\big| \gamma_{\widetilde{Z}_n,Y}(k)\big|^m
 $$
 $$
 \leq \sum^{\infty}_{m=2} \Big(\frac{g^2_{1,n,m}+g^2_{\cos,m}}{2}\Big) \cdot m! (2k_0+1)
 +  \sum^{\infty}_{m=2} \Big(\frac{g^2_{1,n,m}+g^2_{\cos,m}}{2}\Big) \cdot m! \sum_{k_0+1 \leq |k| } \big| \gamma_{Y,\alpha_*}(k)\big|^m
 < \infty,
$$
where finiteness is a consequence of \eqref{e:LHS_summability_cross-cov_bound}. This establishes \eqref{e:summability_cross-cov} for $0 < \alpha < 1$.

When $\alpha = 1$, we can reexpress $\sigma(n)|\gamma_Y(n+k)-\gamma_Y(k)| = \sigma(n)\cdot 1_{\{ k \in \{-n,0\}\}}$. Therefore, $\sum^{\infty}_{k = - \infty} \big| \gamma_{\widetilde{Z}_n,Y}(k)\big|^m < 2$. Hence, the bound \eqref{e:LHS_summability_cross-cov_bound} holds, and so does \eqref{e:summability_cross-cov}.

So, for $0 < \alpha < 3/2$, \eqref{e:summability_cross-cov} holds when $Y$ is a fGn. Now suppose $Y$ is a standard fOU process with parameter $\alpha = 2H \neq 1$, and let $\gamma_{Y,\alpha}$ be its autocovariance function. By expression \eqref{e:gamma_y(k)_fOU}, there exist a parameter value $\alpha_* \in (\max\{\alpha,1\},3/2)$, a fGn correlation function $\gamma_{\textnormal{fGn},\alpha_*}$ and $k_0 \in \bbN$ large enough so that
$$
|\gamma_{Y,\alpha}(k)| < \gamma_{\textnormal{fGn},\alpha_*}(k), \quad k_0 \leq 1 + |k|.
$$
Therefore, by relation \eqref{e:sigma(n)<1},
$$
0 < |\gamma_{\widetilde{Z}_n,Y}(k)| = \sigma(n)\big|\gamma_{Y,\alpha}(n+k) - \gamma_{Y,\alpha}(k)\big| \leq
\left\{\begin{array}{cc}
\gamma_{\textnormal{fGn},\alpha_*}(k), & \textnormal{if } k_0+1 \leq |k| ;\\
1, & \textnormal{if } |k|\leq k_0.\\
\end{array}\right.
$$
Thus, analogously to \eqref{e:LHS_summability_cross-cov_bound}, the left-hand side of \eqref{e:summability_cross-cov} is bounded by
$$
\sum^{\infty}_{m=2} \Big(\frac{g^2_{1,n,m}+g^2_{\cos,m}}{2}\Big) \cdot m! (2k_0+1)
 +  \sum^{\infty}_{m=2} \Big(\frac{g^2_{1,n,m}+g^2_{\cos,m}}{2}\Big) \cdot m! \sum_{k_0+1 \leq |k| } \big| \gamma_{\textnormal{fGn},\alpha_*}(k)\big|^m < \infty,
$$
where finiteness is a consequence of \eqref{e:LHS_summability_cross-cov_bound} (for fGn). This establishes \eqref{e:summability_cross-cov} for the case where $Y$ is a fOU process with $\alpha \neq 1$. The case where $\alpha = 1$ (the traditional OU process) is straightforward.
\end{proof}

\begin{remark}
Note that, if $Y$ is a fGn or a fOU process with parameter $0 < \alpha < 2$ ($\alpha = 2H$), condition \eqref{e:sigma(n)<1} is always satisfied for large enough $n$.
\end{remark}

\noindent{\sc Proof of Proposition \ref{prop:E_1(n)}}: Note that $\widehat{E}_{N,1}(n)$ in \eqref{e:E1} can be rewritten as
\begin{equation*}
\bbC \ni  \widehat{E}_{N,1}(n) = \frac{1}{N-n+1} \sum_{k=0}^{N-n} \cos\bigg( \frac{\widetilde{Z}_n(k)}{\sigma(n)} \bigg) + \hspace{0.5mm}\imag \hspace{0.5mm} \frac{1}{N-n+1} \sum_{k=0}^{N-n} \sin\bigg( \frac{\widetilde{Z}_n(k)}{\sigma(n)} \bigg).
\end{equation*}
For $N > n-1$,
$$
\sqrt{N-n+1} \Big(\mathfrak{R}\widehat{E}_{N,1}(n) - \bbE \cos(  Y(n) - Y(0)) \Big)
$$
$$
= \frac{1}{\sqrt{N-n+1}} \sum_{k=0}^{N-n} [\cos(Y(n+k) - Y(k)) - \bbE \cos(  Y(n) - Y(0)) ]
$$
$$
= \frac{1}{\sqrt{N-n+1}} \sum_{k=0}^{N-n} \bigg[ \cos \bigg(\frac{\widetilde{Z}_n(k)}{\sigma(n)}\bigg) - \bbE \cos\bigg(\frac{\widetilde{Z}_n(k)}{\sigma(n)}\bigg) \bigg].
$$
Then, by condition \eqref{e:gammaZtilde(k)=O(k^(alpha-3))} and Theorem \ref{thm:central}, expression \eqref{e:E_1(n)_real} holds. An analogous reasoning further establishes \eqref{e:E_1(n)_imag}. $\Box$\\

\noindent{\sc Proof of Proposition \ref{prop:E_2(n)}}: We prove ($i$) first. By expanding $\widehat{E}_{N,2}$ and using the stationarity of $Y$, we can rewrite the left-hand side of \eqref{e:e_2(n)alp<3/2} as
 $$
\sqrt{N+1} \bigg\{ \abs{ \bbE \cos Y(0)  + \frac{1}{N+1} \sum_{k=0}^{N} [\cos Y(k)  - \bbE \cos Y(k)] }^2
 $$
 $$
 + \abs{\frac{1}{N+1} \sum_{k=0}^{N} \sin Y(k)}^2  - \abs{\bbE \cos Y(0)}^2\bigg\}
$$
$$
 = 2 \hspace{1mm}\bbE \cos Y(0)  \frac{1}{ \sqrt{N+1} } \sum_{k=0}^{N} [\cos Y(k)  - \bbE \cos Y(k) ]
$$
\begin{equation}\label{e:e_2(n)asym1}
+ \frac{1}{ \sqrt{N+1} } \bigg(\frac{1}{ \sqrt{N+1} } \sum_{k=0}^{N} [\cos Y(k) - \bbE \cos Y(k) ] \bigg)^2
  + \bigg( \frac{1}{(N+1)^{3/4}} \sum_{k=0}^{N} \sin Y(k) \bigg)^2.
\end{equation}
We will show that, as $N\rightarrow \infty$, the first term in the sum \eqref{e:e_2(n)asym1} converges to a non-degenerate random variable in distribution, and that the second and third terms in \eqref{e:e_2(n)asym1} converge to zero in probability. Note that, when $0<\alpha < 3/2$, condition \eqref{e:gammaY(k)=O(k^(alpha - 2))} implies that
$$
  \sum_{k=1}^{\infty} \abs{\gamma_Y(k)}^2 < \infty.
$$
By Theorem \ref{thm:central},
\begin{equation}\label{e:e_2(n)cosalp<3/2}
 \frac{1}{\sqrt{N+1}} \sum_{k=0}^{N} [ \cos Y(k) - \bbE \cos Y(k) ] \overset{d}{\rightarrow} \sigma_3 B_3(1), \quad N \rightarrow \infty,
\end{equation}
where
  $$
  \sigma^2_3 = \sum_{m=2}^{\infty} g_{\cos,m}^2 m! \sum_{k=-\infty}^{\infty} (\gamma_Y(k))^m.
  $$
Since $\frac{1}{\sqrt{N+1}} \rightarrow 0$, by \eqref{e:e_2(n)cosalp<3/2} and Slutsky's theorem, the second term in the sum \eqref{e:e_2(n)asym1} converges to zero in probability. As for the third term in the sum \eqref{e:e_2(n)asym1}, when $0<\alpha\leq 1$, condition \eqref{e:gammaY(k)=O(k^(alpha - 2))} implies that
$$
\sum_{k=1}^{\infty} \abs{\gamma_Y(k)} < \infty.
$$
Thus, by Theorem \ref{thm:central},
\begin{equation}\label{e:e_2(n)sinalp<=1}
\frac{1}{\sqrt{N+1}} \sum_{k=0}^{N} \sin Y(k) \overset{d}{\rightarrow} \sigma_4 B_4(1), \quad N \rightarrow \infty,
\end{equation}
where
  $$
  \sigma^2_4 = \sum_{m=1}^{\infty} g_{\sin,m}^2 m! \sum_{k=-\infty}^{\infty} (\gamma_Y(k))^m.
  $$
Then, the third term in the sum \eqref{e:e_2(n)asym1} satisfies
\begin{equation}\label{e:third_term_0<alpha=<1}
\frac{1}{ \sqrt{N+1} } \bigg( \frac{1}{ \sqrt{N+1} } \sum_{k=0}^{N} \sin Y(k) \bigg)^2 \overset{\bbP}{\rightarrow} 0, \quad N\rightarrow\infty,
\end{equation}
which is a consequence of \eqref{e:e_2(n)sinalp<=1} and Slutsky's theorem. On the other hand, suppose $1<\alpha <3/2$. Under condition \eqref{e:gammaY(k)=O(k^(alpha - 2))}, by Theorem \ref{thm:noncentral},
\begin{equation}\label{e:e_2(n)sinalp<3/2}
\frac{1}{\sqrt{L(N+1)} (N+1)^{\alpha/2}} \sum_{k=0}^{N} \sin Y(k)  \overset{d}{\rightarrow} g_1 \beta_{1,H} Z_H^{(1)}(1).
 \end{equation}
Then, the third term in the sum \eqref{e:e_2(n)asym1} satisfies
\begin{equation}\label{e:third_term_1<alpha<3/2}
\frac{L(N+1)}{(N+1)^{3/2-\alpha}}\bigg( \frac{1}{\sqrt{L(N+1)} (N+1)^{\alpha/2}} \sum_{k=0}^{N} \sin Y(k) \bigg)\overset{\bbP}{\rightarrow} 0, \quad N\rightarrow\infty,
\end{equation}
which results from \eqref{e:e_2(n)sinalp<3/2} and Slutsky's theorem. Thus, expression \eqref{e:e_2(n)alp<3/2} follows.

We now prove ($ii$). So, suppose $3/2<\alpha<2$ and rewrite the left-hand side of \eqref{e:e_2(n)alp>3/2} as
 $$
\frac{(N+1)^{2-\alpha}}{ L(N+1)} \bigg\{ \abs{ \bbE \cos Y(0) + \frac{1}{N+1} \sum_{k=0}^{N} [\cos Y(k) - \bbE \cos Y(k)] }^2
 $$
 $$
 + \abs{\frac{1}{N+1} \sum_{k=0}^{N} \sin Y(k) }^2  - \abs{\bbE \cos Y(0) }^2\bigg\}
$$
$$
 = 2 \hspace{0.5mm}\bbE \cos Y(0)  \frac{1}{ L(N+1) (N+1)^{\alpha - 1}} \sum_{k=0}^{N} [\cos Y(k) - \bbE \cos Y(k)]
$$
$$
+ \frac{ L(N+1)}{(N+1)^{2-\alpha}} \bigg(\frac{1}{ L(N+1)(N+1)^{\alpha - 1} } \sum_{k=0}^{N} [\cos Y(k) - \bbE \cos Y(k)] \bigg)^2
$$
\begin{equation}\label{e:e_2(n)asym1alp>3/2}
  + \bigg( \frac{1}{\sqrt{L(N+1)}(N+1)^{\alpha/2}} \sum_{k=0}^{N} \sin Y(k) \bigg)^2.
\end{equation}
Under condition \eqref{e:gammaY(k)=O(k^(alpha - 2))}, by Theorem \ref{thm:noncentral_multi},
$$
\bigg( \frac{1}{L(N+1) (N+1)^{\alpha - 1}} \sum_{k=0}^{N} [\cos Y(k) - \bbE \cos Y(k)],
    \frac{1}{\sqrt{L(N+1)} (N+1)^{\alpha/2}} \sum_{k=0}^{N} \sin Y(k) \bigg)^{\top}
$$
\begin{equation}\label{e:e_2(n)alp>3/2multi}
  \overset{d}{\rightarrow} \Big( g_{1,2}\beta_{2,\alpha - 1} \widehat{I}_{2}(f_{\alpha - 1,2,1}),
     g_{2,1}\beta_{1,\alpha/2} \widehat{I}_{1}(f_{\alpha/2,1,1})\Big)^{\top}.
\end{equation}
By Slutsky's theorem and \eqref{e:e_2(n)alp>3/2multi}, the second term in the sum \eqref{e:e_2(n)asym1alp>3/2} converges to zero in probability. Then, by \eqref{e:e_2(n)alp>3/2multi}, relation \eqref{e:e_2(n)alp>3/2} holds. $\Box$\\

\noindent{\sc Proof of Theorem \ref{t:E(n)dist}}: We prove ($i.1$) first. Note that, when $0<\alpha < 3/2$, we can express
$$
\sqrt{N+1} \big\{\widehat{E}_{N,2} - |\bbE \cos Y(0)|^2 \big\}
= 2 \hspace{0.5mm}\bbE \cos Y(0) \frac{1}{\sqrt{N+1}} \sum^{N}_{k=0} \big[ \cos Y(k) - \bbE \cos Y(k)\big] + o_{\bbP}(1)
$$
(see expressions \eqref{e:e_2(n)asym1}, \eqref{e:e_2(n)cosalp<3/2}, \eqref{e:third_term_0<alpha=<1} and \eqref{e:third_term_1<alpha<3/2}). Thus, by Propositions \ref{prop:E_1(n)} and \ref{prop:E_2(n)} and Slutsky's theorem,
$$
\sqrt{N+1} \Big( \mathfrak{R} \widehat{E}_N(n) - \bbE \cos\big(Y(n) - Y(0)\big) + \abs{\bbE \cos Y(0)}^2 \Big)
$$
$$
= \frac{\sqrt{N+1}}{\sqrt{N-n+1}} \sqrt{N-n+1}\Big(\mathfrak{R} \widehat{E}_{N,1}(n) - \bbE \cos\big(Y(n) - Y(0)\big)\Big)
- \sqrt{N+1} \big[\widehat{E}_{N,2} - \abs{\bbE \cos Y(0)}^2 \big]\Big)
$$
$$
\overset{d}{\rightarrow} \sigma_1 B_1(1) + 2 \big(\bbE \cos Y(0)\big) \sigma_3 B_3(1),
$$
where $(B_1(1),B_3(1))^{\top}$ is a Gaussian vector since $(Y,\widetilde{Z}_n)$ is a Gaussian vector. In addition, let $G_1(x) = \cos(x/\sigma(n)), G_2(x) = \cos(x)$. By Proposition \ref{prop:g1g2cov},
$$
\cov \Big(\cos \Big(\frac{\widetilde{Z}_n(k+j)}{\sigma(n)}\Big), \cos Y(j) \Big)= \sum_{m = 2}^{\infty} g_{1,n,m} \hspace{0.5mm}g_{\cos,m} \hspace{0.5mm} m! (\gamma_{\widetilde{Z}_n,Y}(k))^m,
$$
since $g_{1,n,1}=g_{\cos,1}=0$. Hence, by condition \eqref{e:cross-cov_summability},
$$
(N+1)\cdot \cov \Big( \frac{1}{N-n+1} \sum^{N-n}_{k=0} \cos\Big( \frac{\widetilde{Z}_{n}(k)}{\sigma(n)}\Big),  -\frac{2 \hspace{0.5mm}\bbE \cos Y(0)}{N+1} \sum^{N}_{k'=0}\cos Y(k')   \Big)
$$
$$
\sim -\frac{2 \hspace{0.5mm}\bbE \cos Y(0)}{N+1}\sum^{N}_{k=0} \sum^{N}_{k'=0}\cov \Big(   \cos\Big( \frac{\widetilde{Z}_{n}(k)}{\sigma(n)}\Big),\cos Y(k')\Big)
$$
$$
\rightarrow - 2 \hspace{0.5mm}\bbE \cos Y(0) \sum^{\infty}_{\ell=-\infty}\sum_{m = 2}^{\infty} g_{1,n,m} \hspace{0.5mm}g_{\cos,m}\hspace{0.5mm} m! (\gamma_{\widetilde{Z}_n,Y}(\ell))^m, \quad N \rightarrow \infty
$$
(e.g., \cite{abry:didier:2018}, Lemma B.3). Therefore, by the analyticity of all characteristic functions involved, expression \eqref{e:corrB1B3} holds. This establishes ($i$).

Now we show ($ii$). When $3/2<\alpha<2$, by Propositions \ref{prop:E_1(n)} and \ref{prop:E_2(n)}, and Slutsky's theorem,
$$
\frac{(N+1)^{2-\alpha}}{L(N+1)} \Big( \mathfrak{R} \widehat{E}_N(n) - \bbE \cos\big(Y(n) - Y(0)\big) + \abs{\bbE \cos Y(0)}^2 \Big)
$$
$$
= \frac{(N+1)^{2-\alpha}}{L(N+1)\sqrt{N-n+1}}  \sqrt{N-n+1}\Big( \mathfrak{R} \widehat{E}_{N,1}(n) - \bbE \cos\big(Y(n) - Y(0)\big) \Big)
$$
$$
- \frac{(N+1)^{2-\alpha}}{L(N+1)} [\widehat{E}_{N,2} - \abs{\bbE \cos Y(0)}^2]
$$
$$
\overset{d}{\rightarrow} g_{1,2}\beta_{2,\alpha - 1} \widehat{I}_{2}(f_{\alpha - 1,2,1})
    + g_{2,1}^2 \beta_{1,\alpha/2}^2 \widehat{I}_{1}^2(f_{\alpha/2,1,1}).
$$
Thus, ($i.2$) holds.

Statement ($ii$) is a consequence of the fact that $\mathfrak{I} \widehat{E}_N(n) = \mathfrak{I} \widehat{E}_{N,1}(n)$, as well as of Proposition \ref{prop:E_1(n)} and Slutsky's theorem.

To establish expression \eqref{e:Cov(B2,Bell(1))=0}, first note that $\cos(x)$ is an even function while $\sin(x)$ is an odd function. Thus, $g_{1,n,2m+1}= g_{2,n,2m} = 0$, $m \in \bbN \cup \{0\}$. Then, by Proposition \ref{prop:g1g2cov}, $B_1(1)$ and $B_2(1)$ in Proposition \ref{prop:E_1(n)} are two independent standard normal variables, since they are Gaussian and uncorrelated. By a similar reasoning, $B_2(1)$ and $B_3(1)$ in Propositions \ref{prop:E_1(n)} and \ref{prop:E_2(n)}, respectively, are two independent standard normal variables, since they are Gaussian and uncorrelated. Thus, \eqref{e:Cov(B2,Bell(1))=0} holds, as claimed. $\Box$\\

\section{Pseudocode}\label{s:pseudocode}

{\small
\begin{center}
\begin{tabular}{|l|}
\hline
\multicolumn{1}{|c|}{\textbf{$\widehat{E}_N(n)$-based hypothesis testing} (see Section \ref{s:discussion})}\\
\\
\hline
\\
\textbf{Input}: \\
$\bullet$ one observed particle path ${\mathbf Y} = \{Y_1,Y_2,\hdots,Y_N\}_{N \in \bbN}$ of length $N$;\\
$\bullet$ a preset significance level $\eta \in [0,1]$;\\
\\
\textbf{Step 1}: use ${\mathbf Y}$ to compute an estimator $\widehat{\alpha}$ of $\alpha$ by means of the $\MSD$;\\
\\
\textbf{Step 2}: estimate $\mu_{{\mathfrak R}}$ by means of relation \eqref{e:Re_part_centering_estimator}; \\
\\
\textbf{Step 3}: estimate $\theta^2_{{\mathfrak R}}$ and $\theta^2_{{\mathfrak I}}$ by means of
relations \eqref{e:theta^2_Im_estimator} and \eqref{e:theta^2_Re_estimator};\\
\\
\textbf{Step 4}: compute ${\mathcal W}^*_N := \Big\{\frac{\sqrt{N+1} \big(\hspace{0.5mm}\mathfrak{R} \widehat{E}_N(n)- \widehat{\mu}_{{\mathfrak R}}\big)}{\widehat{\theta}_{{\mathfrak R}}} \Big\}^2+ \Big\{\frac{\sqrt{N+1}\hspace{0.5mm} \mathfrak{I} \widehat{E}_N(n)}{\widehat{\theta}_{{\mathfrak I}}}\Big\}^2$;\\
\\
\textbf{Step 5}: if ${\mathcal W}^*_N > \chi^2_2(\eta)$, reject the null hypothesis of mixing; otherwise, retain it.\\
\\
\hline
\end{tabular}
\end{center}
}

\section{Data availability statement}

Data sharing is not applicable to this article as no new data were created or analyzed in this study.

\bibliographystyle{plain}
\bibliography{mixing_asymptotics}

\small

\bigskip

\noindent \begin{tabular}{lr}
Kui Zhang and Gustavo Didier \\
Mathematics Department \\
Tulane University  \\
6823 St.\ Charles Avenue \\
New Orleans, LA 70118, USA \\
\texttt{kzhang3@tulane.edu} \\
\texttt{gdidier@tulane.edu}  \\
\end{tabular}\\

\smallskip

\end{document}